\documentclass[12pt,oneside]{amsart}
\usepackage{amsmath,amssymb,latexsym,soul,cite,mathrsfs}
\usepackage[pagewise]{lineno}
\usepackage{color,enumitem,graphicx}
\usepackage[colorlinks=true,urlcolor=blue,
citecolor=red,linkcolor=blue,linktocpage,pdfpagelabels,
bookmarksnumbered,bookmarksopen]{hyperref}
\usepackage[english]{babel}
\usepackage{lineno}
\usepackage[left=2.8cm,right=2.8cm,top=3cm,bottom=3cm]{geometry}

\usepackage[hyperpageref]{backref}

\numberwithin{equation}{section}
\newtheorem{theorem}{Theorem}[section]
\newtheorem{lemma}[theorem]{Lemma}
\newtheorem{corollary}[theorem]{Corollary}
\newtheorem{proposition}[theorem]{Proposition}

\newtheorem{remark}[theorem]{Remark}

\renewcommand{\epsilon}{\varepsilon}

\renewcommand{\rightarrow}{\to}

\begin{document}
\title[Sobolev-type inequality and minimizers]{On a Sobolev-type inequality  and its minimizers}

\author[J.F.\ de Oliveira]{Jos\'{e} Francisco de Oliveira}\thanks{First author was partially supported by  CNPq grant number 309491/2021-5}
\address[J.F.\ de Oliveira]{
\newline\indent Department of Mathematics
	\newline\indent 
	Federal University of Piau\'{i}
	\newline\indent
	64049-550 Teresina, PI, Brazil}
	\email{\href{mailto:jfoliveira@ufpi.edu.br}{jfoliveira@ufpi.edu.br}}
 \author[J.N.\ Silva ]{Jeferson Silva}
\address[J.N.\ Silva]{\newline\indent Department of Mathematics
	\newline\indent 
	Federal University of the Delta of Parna\'{i}ba
	\newline\indent
	CEP 64202-020, Parna\'{i}ba, PI, Brazil}
	\email{\href{mailto:j.n.silva@ufpi.edu.br}{j.n.silva@ufpi.edu.br}}
 

\begin{abstract}
Critical Sobolev-type inequality for a class of weighted Sobolev spaces on the entire space is established. We also investigate the existence of extremal function for the associated variational problem. As an application, we prove the existence of a weak solution for a general class of critical semilinear elliptic equations related to the polyharmonic operator.  
\end{abstract}

\maketitle

\section{Introduction and main results}

Let $AC_{loc}(0,R)$ be the set of all locally absolutely continuous functions on interval $(0,R)$.  For each non-negative  integer $\ell$, let $AC^{\ell}_{loc}(0,R)$ be the set of all functions $u:(0,R)\to\mathbb{R}$
such that  $u^{(\ell)}\in AC_{loc}(0,R)$, where  $u^{(\ell)}=d^{\ell}u/dr^{\ell}$. For any $m\ge 1$ integer number, we define
\begin{equation}\label{testfunctions-right}
    AC^{m-1}_{\mathrm{R}}(0,R)=\big\{u\in AC^{m-1}_{loc}(0,R)\,:\, \lim_{r\to R} u^{(j)}(r) = 0,\; j=0,\cdots, m-1\big\}.
\end{equation}
For $0<R\leq \infty$, $p\ge 1$ and $\theta\in\mathbb{R}$, we denote by $L^{p}_{\theta}=L^{p}_{\theta}(0,R)$ the weighted Lebesgue space of the Lebesgue measurable functions
	$u :(0,R)\rightarrow\mathbb{R}$ endowed with the norm 
	$$
	\|u\|_{L^{p}_{\theta}}=\left(\int_{0}^{R}|u(r)|^{p}r^{\theta}\mathrm{d}r \right)^{\frac{1}{p}}<\infty.
    $$
Now, for real numbers $\alpha_{j}> -1$ for $j=0,1, \cdots, m$,  we denote $W^{m,p}_{R}=W^{m,p}_{R}(\alpha_{0}, \cdots, \alpha_{m})$ the weighted Sobolev space which  consists of all functions $u\in AC^{m-1}_{loc}(0,R)$ such that $u^{(j)} \in L^{p}_{\alpha_{j}}$, for all $j = 0, 1, \cdots, m.$ Then, $W^{m,p}_{R}$ is a Banach space endowed with the norm
\begin{equation}\label{c0-norma1}
    \|u\|_{W^{m, p}_{R}} = \left( \sum_{j=0}^{m} \| u^{(j)}\|_{L^{p}_{\alpha_{j}}}^{p} \right)^{\frac{1}{p}}.
\end{equation}
Further, we denote by $X^{m, p}_{R} = X^{m,p}_{R}(\alpha_{0}, \alpha_{1}, \cdots, \alpha_{m})$ the closure of set 
$$
W_{0,R}=W_{0,R}(\alpha_{0}, \cdots, \alpha_{m})=AC^{m-1}_{\mathrm{R}}(0,R)\cap W^{m,p}_{R}(\alpha_{0}, \cdots, \alpha_{m})
$$
under the norm \eqref{c0-norma1}. In the bounded case  $0<R<\infty$, if the weights $\alpha_0, \cdots , \alpha_m$ safisfy
	 the transition condition 
	\begin{equation}\label{trasition-C}
	    \alpha_{j-1}\ge \alpha_{j}-p,\quad j=1,\cdots, m,
	\end{equation}
we observe that the norm \eqref{c0-norma1} and 
\begin{equation}\label{norma-diric}
    \|u^{(m)}\|_{L^{p}_{\alpha_{m}}} = \left( \int^{R}_{0} |u^{(m)}|^{p} r^{\alpha_{m}}dr\right)^{\frac{1}{p}},
\end{equation}
are equivalent  on $X^{m,p}_{R}(\alpha_0,\cdots, \alpha_m)$,
see \cite{JAJ} for more details.

For each $\alpha>-1$, we define the $\alpha$-generalized Laplacian operator in the radial form 
\begin{equation}\label{laplace-geral}
    \Delta_{\alpha} u = r^{-\alpha}(r^{\alpha} u')',
\end{equation}
where $u\in C^2(0,R)$. The operator \eqref{laplace-geral} appears in \cite{JAJ} and we observe that, for $\alpha=N-1$ positive integer number, $\Delta_{\alpha} u $ agrees precisely with the Laplacian operator acting on radial symmetric functions $u$  defined on a ball centered at the origin $B(0,R)\subset\mathbb{R}^{N}, R>0$. To deal with  the $\alpha$-generalized Laplacian operator, we introduce the set
$$D_{0,R}(\alpha)=D_{0,R}(\alpha,m,p)=\big\{u\in AC^{m-1}_{\mathrm{R}}(0,R)\; :\; u^{(m)}\in L^{p}_{\alpha}\big\}.$$ 
Finally, we define  the new space $\mathcal{D}^{m,p}_{R}(\alpha)$ by completion of $D_{0,R}(\alpha,m,p)$ under the norm \eqref{norma-diric} with $\alpha=\alpha_m$. Another norm we can consider on $\mathcal{D}^{m,p}_{R}(\alpha)$ is the \textit{Navier norm} 
\begin{equation}\label{norma-grad}
    \|u\|_{\nabla^{m}_{\alpha}}= \|\nabla^{m}_{\alpha} u\|_{L^{p}_{\alpha}}
\end{equation}
where, for $u\in D_{0,R}(\alpha,m,p)$
$$\nabla^{m}_{\alpha} u=\left\{\begin{array}{lc}
\Delta_{\alpha}^{\frac{m}{2}} u,&\mbox{if}\quad m\,\, \text{is even},\\
\left(\Delta_{\alpha}^{\frac{m-1}{2}} u \right)', &\mbox{if}\quad m\,\, \text{is odd}
\end{array}\right.$$
is the $\alpha$-generalized $m$-th order gradient. 
 We observe that, as we will see in Proposition~\ref{Equi-norms} below, the norms \eqref{norma-diric} and \eqref{norma-grad} are equivalent, provided that $\alpha-mp+1>0$. For this case, we can alternatively define  $\mathcal{D}^{m,p}_{R}(\alpha)$ to be
 the completion of $D_{0,R}(\alpha)$ under the norm \eqref{norma-grad}.

\noindent We would like to mention that the weighted Sobolev spaces $W^{m, p}_{R}$, $X^{m, p}_{R}$ and $\mathcal{D}^{m,p}_{R}(\alpha)$ hide surprises and some interesting points have been drawing attention since the pioneer work due to P. Cl\'{e}ment at al. \cite{Clement-deFigueiredo-Mitidieri}, see for instance \cite{Ibero1,Abreu,doOdeOliveira2014, JAJ,DCDS2019,CV2023, DoLuHa}. 

\noindent By using  Hardy-type inequalities in \cite{Opic}, Cl\'{e}ment at al. \cite[Proposition~1.0]{Clement-deFigueiredo-Mitidieri} were able to show that
	for the first order derivative $m=1$, the behavior of the functions in $u\in X^{1,p}_R(\alpha_0,\alpha_1)$ is driven by the parameters $\alpha_1$ and $p$.
 In fact, we can distinguish three  important cases depending on the conditions on $\alpha_1$ and $p$:

 \paragraph*{I}\textit{Sobolev:}
	$\alpha_1-p+1>0$. 
 
 \paragraph*{II}\textit{Trudinger-Moser:}  $\alpha_1-p+1=0.$
	
 \paragraph*{II}\textit{Morrey:}  $\alpha_1-p+1<0$.
 
In the bounded case $0<R<\infty$, if we assume the \textit{Sobolev} condition, according to \cite{Clement-deFigueiredo-Mitidieri}, we get the  continuous embedding
\begin{equation}\label{eq10}
	X^{1,p}_R(\alpha_0,\alpha_1)\hookrightarrow L^{q}_{\theta},  \;\; \mbox{if}\;\;  1 < q\leq  p^{*} \;\;\mbox{and}\;\; \min\left\{\theta,\alpha{_0}\right\}\ge \alpha_1-p,
	\end{equation}
where 
$$p^*=\frac{(\theta+1)p}{\alpha_1-p+1}$$ 
represents  the critical Sobolev exponent of $X^{1,p}_R(\alpha_0,\alpha_1)$. In addition, in the strict case $q<p^{*}$, the embedding \eqref{eq10} is compact.

In contrast, for the \textit{Trudinger-Moser} condition  we have the compact embedding 
	\begin{equation}\label{TM-embeddins1}
	X^{1,p}_R(\alpha_0,\alpha_1)\hookrightarrow L^{q}_{\theta}, \quad\mbox{if}\quad 1 < q < \infty \quad\mbox{and}\quad \theta>-1.
	\end{equation}
In fact, in this case, the	 maximal growth function $\phi(s), s\in\mathbb{R}$ such that $\phi\circ u\in L^{1}_{\theta}$,  for all $u\in  X^{1,p}_R(\alpha_0,\alpha_1) $
is of the exponential-type, see  \cite{CCM02017}.

Lastly, for the \textit{Morrey} condition  $\alpha_1-p+1<0$,  the continuous embedding 
 \begin{equation}\label{M-embedding}
    X^{1,p}_R(\alpha_0,\alpha_1)\hookrightarrow  C^{0,\sigma}([0,R]), \;\; \mbox{with}\;\;\sigma=\min\big\{1-\frac{\alpha_1+1}{p}, 1-\frac{1}{p}\big\} 
 \end{equation}
 was recently proved in \cite{DoLuHa}. 
 
Extensions of the embeddings \eqref{eq10}, \eqref{TM-embeddins1} and \eqref{M-embedding} for higher order derivatives $m\ge 2$  can be founded in \cite{DoLuHa} and \cite{JAJ}, but only for the bounded case $0<R<\infty$. In particular, according to \cite[Theorem 1.1]{JAJ}, for $0<R<\infty$, if we assume the  \textit{Sobolev} condition
 \begin{equation}\label{m-Sobolev}
 \alpha_m-mp+1>0
 \end{equation}
 then we obtain the continuous embedding
\begin{equation}\label{Xe77}
    X^{m,p}_{R} \hookrightarrow L_{\theta}^{q}, \quad \forall\,\, 1<q\leq p^{*} \quad \text{and}\quad \min \{\theta, \alpha_{j-1}\} \geq \alpha_{j} -p,\;\; j= 1, \cdots, m
\end{equation}
where 
$$p^{*} = p^{*}(\theta, p, m)=\dfrac{(\theta+ 1)p}{\alpha_{m} -mp +1}$$
is the critical Sobolev exponent of $ X^{m,p}_{R}$. Further, \eqref{Xe77} is compact in the strict case $q<p^*$.

In this work, we are mainly interested  in the \textit{Sobolev} case  for unbounded domain $R=\infty$. 
Our first result  reads below.

\begin{theorem}\label{c0-t2}
   Let $p\geq 2$, $\alpha, \theta>-1 $, $0<R\le \infty$ and let $m\in \mathbb{Z}_{+}$ such that 
   \begin{equation}\label{c0-cond1}
    \theta\ge \alpha-mp\;\;\mbox{and}\;\; \alpha-mp+1>0.
\end{equation}
Then there exists $C>0$ such that
\begin{equation}\label{c0-t2-eq}
     \left( \int_{0}^{R} |u(r)|^{p^{*}} r^{\theta} \, dr\right)^{\frac{1}{p^{*}}} \leq C \left( \int_{0}^{R} |\nabla^{m}_{\alpha} u(r)|^{p} r^{\alpha} \, dr \right)^{\frac{1}{p}},\;\;\mbox{for all}\;\; u \in \mathcal{D}^{m,p}_{R}(\alpha),
\end{equation}
where
  $$p^{*} = \dfrac{(\theta +1)p}{\alpha -mp +1}.$$ In particular,  we obtain the continuous embedding 
  \begin{equation}\label{e7full}
      \mathcal{D}^{m,p}_{R}(\alpha) \hookrightarrow L^{p^{*}}_{\theta}.
  \end{equation}
\end{theorem}
\noindent It is worth mentioning that the estimate \eqref{c0-t2-eq} holds for either $0<R<\infty$ or $R=\infty$. In addition,  for each $0<R\le \infty$ we can define 
\begin{equation}\label{c0-e40}
  \mathcal{S}=\mathcal{S}(m,p,\alpha, \theta, R)  = \inf \left\{ \|\nabla_{\alpha}^{m} u \|^{p}_{L_{\alpha}^{p}}\, : \,\, u \in \mathcal{D}^{m,p}_{R}(\alpha),\; \|  u \|^{p}_{L_{\theta}^{p^{*}}}=1\right\}.
\end{equation}
Then, we have $\mathcal{S}>0$ and  $\mathcal{S}^{-\frac{1}{p}}$ is the \textit{best constant} for the Sobolev type embedding \eqref{e7full}. An interesting question in variational calculus is knowing when the best constant for critical Sobolev-type inequality is attained or not, see \cite{EFJ,S,WangWillem, Gazzola, Talenti} and references therein. In the case of $ \eqref{c0-t2-eq}$, as we will see in Lemma~\ref{a0-prop1} below, the constant  $\mathcal{S}(m,p,\alpha, \theta, R)$ does not depend on $R>0$. In addition, we are able to prove that $\mathcal{S}(m,2,\alpha, \theta, \infty)$ is attained. In order to obtain our attainability result we adopt the concentration-compactness argument due to P.L. Lions  \cite{PL3,PL4, WangWillem, WIL} for $\mathcal{D}^{m,2}_{\infty}(\alpha)$ to deal with the loss of compactness due to the critical exponent $2^{*} = 2(\theta + 1)/(\alpha -2m +1)$ and  the unbounded domain $R=\infty$, see Lemma~\ref{a0-prop2} below. Actually, we will prove the following:
\begin{theorem}\label{a0-teo1} Let $\alpha, \theta>-1 $ and let $m\in \mathbb{Z}_{+}$ such that $\theta\ge \alpha-2m$ and $\alpha-2m+1>0$. Then, 
 there exists $z\in \mathcal{D}^{m,2}_{\infty}(\alpha)$ such that $$\|z\|_{L^{2^{*}}_{\theta}} =1\;\mbox{and}\;\|\nabla^{m}_{\alpha} z\|^{2}_{L^{2}_{\alpha}} = \mathcal{S}(m,2,\alpha, \theta, \infty).$$
\end{theorem}
\noindent It is well known that the Sobolev-type inequalities and their extremals are basic tools in many aspects of mathematical analysis and partial differential equations. As application of  Theorem~\ref{c0-t2} and Theorem~\ref{a0-teo1} we provide the following existence result for a general class of elliptic equations driven by 
polyharmonic operator.

\begin{corollary}\label{d1} For $\alpha, \theta$ and $m$ under the assumptions of the Theorem~\ref{a0-teo1}, there exists an weak solution  $z\in \mathcal{D}^{m,2}_{\infty}(\alpha)$ to the critical semilinear equation 
     \begin{equation}\label{c0-e8}
    (-\Delta_{\alpha})^{m} u = r^{\theta-\alpha} |u|^{2^{*}-2} u \,\, \text{in}\,\, (0, \infty)
\end{equation}
where $\Delta_{\alpha}$ is given by \eqref{laplace-geral} and $2^{*} = 2(\theta + 1)/(\alpha -2m +1)$ is the critical exponent.
\end{corollary}

The paper is organized as follows. Section~\ref{sec2} is devoted to the proof of
Theorem~\ref{c0-t2} by argument based on a class of Hardy-type inequalities, see Theorem~\ref{Teorema hardy} below. The attainability result Theorem~\ref{a0-teo1} and Corollary~\ref{d1} are proven in Section~\ref{sec3}.

\section{Critical Sobolev-type inequality on the entire space}
\label{sec2}
In this section we prove the critical Sobolev-type inequality  stated in Theorem \ref{c0-t2}. 
The following Hardy-type inequality is discussed by Kufner-Opic \cite[Chap. 1. Sec.10, pg. 142] {Opic}. See also, \cite[Theorem 4.3, Remark 4.4]{KP}.
\begin{theorem}\label{Teorema hardy}
    Let $1<p\le q < \infty$, $0<R\le\infty$, $m\in \mathbb{Z}_{+}$ and $v, z$ measurable and positive functions in $(0, R)$. Then there is $C>0$ such that 
    \begin{equation}\label{Hardy-m ordem}
     \left( \int_{0}^{R} |u(r)|^{q} z(r) \, dr\right)^{\frac{1}{q}} \leq C \left( \int_{0}^{R} |u^{(m)}(r)|^{p} v(r) \, dr \right)^{\frac{1}{p}},   
    \end{equation}
holds,  for all $u \in AC^{m-1}_{\mathrm{R}}(0, R)$ if and only if $$\mathcal{A} = \max \{ \mathcal{A}_{m, 0}, \mathcal{A}_{m, 1}\} < \infty,$$ where  
\begin{eqnarray*}
    \mathcal{A}_{m, 0} &=&  \sup_{0<t<R} \bigg( \int_{0}^{t} (t-r)^{(m-1)q} z(r) \, dr\bigg)^{\frac{1}{q}} \bigg( \int_{t}^{R} v^{-\frac{1}{p-1}}(r) \, dr\bigg)^{\frac{p-1}{p}};
    \\
\mathcal{A}_{m, 1} &=& \sup_{0<t<R} \bigg( \displaystyle\int_{0}^{t}  z(r) \, dr\bigg)^{\frac{1}{q}}  \bigg( \int_{t}^{R} (r-t)^{\frac{p(m-1)}{p-1}} v^{-\frac{1}{p-1}}(r) \, dr\bigg)^{\frac{p-1}{p}}.
\end{eqnarray*}
\end{theorem}
\begin{remark}\label{Remar-left}
   As observed in \cite{KP},  in a similar way we can describe necessary and sufficient conditions to make equation \eqref{Hardy-m ordem} true for the set 
   \begin{equation}\nonumber
    AC^{m-1}_{\mathrm{L}}(0,R)=\big\{u\in AC^{m-1}_{loc}(0,R)\, :\, \lim_{r \to 0} u^{(j)}(r) = 0,\; j=0,1,\cdots, m-1\big\}.
\end{equation}
In fact, in this case,   \eqref{Hardy-m ordem} hold for $1<p\le q < \infty$ if and only if
   $$\tilde{\mathcal{A}} = \max \{ \tilde{\mathcal{A}}_{m, 0}, \tilde{\mathcal{A}}_{m, 1}\} < \infty,$$ where  
\begin{eqnarray*}
    \tilde{\mathcal{A}}_{m, 0} &=&  \sup_{0<t<R} \bigg( \int_{t}^{R} (r-t)^{(m-1)q} z(r) \, dr\bigg)^{\frac{1}{q}} \bigg( \int_{0}^{t} v^{-\frac{1}{p-1}}(r) \, dr\bigg)^{\frac{p-1}{p}};
    \\
\tilde{\mathcal{A}}_{m, 1} &=& \sup_{0<t<R} \bigg( \displaystyle\int_{t}^{R}  z(r) \, dr\bigg)^{\frac{1}{q}} \bigg( \int_{0}^{t} (t-r)^{\frac{p(m-1)}{p-1}} v^{-\frac{1}{p-1}}(r) \, dr\bigg)^{\frac{p-1}{p}}.
\end{eqnarray*}
\end{remark}
\noindent As by-product of Theorem~\ref{Teorema hardy} and Remark~\ref{Remar-left} we have the following:
\begin{proposition}\label{c0-prop1}
Let $0<R\le \infty$, $p>1$, $m\in \mathbb{Z}^{*}_{+}$, and $\gamma,\theta\in\mathbb{R}$ such that $\gamma - mp +1 \not =0.$ 
Then the inequality
\begin{equation}\label{c0-eprop1}
     \left( \int_{0}^{R} |u(r)|^{p^{*}} r^{\theta} \, dr\right)^{\frac{1}{p^{*}}} \leq C \left( \int_{0}^{R} |u^{(m)}(r)|^{p} r^{\gamma} \, dr \right)^{\frac{1}{p}}, \; \mbox{with}\; p^{*}= \frac{(\theta +1)p}{\gamma -mp +1}
\end{equation}
holds for some $C>0$ (independent of $u$) under the conditions:
\begin{enumerate}
    \item [$(a)$] for all $u \in AC^{m-1}_{\mathrm{R}}(0, R)$ if $\theta\ge \gamma-mp$ and  $\gamma - mp +1 > 0$. 
    \item[$(b)$] for all $u \in AC^{m-1}_{\mathrm{L}}(0, R)$ if $\theta\le \gamma-mp$ and  $\gamma - p +1 <0$.
\end{enumerate}
\end{proposition}
\begin{proof}
   $(a)$ For $\theta\ge \gamma-mp$ and  $\gamma- mp +1 > 0$ we have $1<p\le p^{*}$. Then we are in a position to apply Theorem~\ref{Teorema hardy} with $z(r) = r^{\theta}$ and $v(r)= r^{\gamma}$. For any $0<r<t<R\le \infty$ we have $(t-r)^{(m-1)p^*}\le t^{(m-1)p^*}$, thus 
$$
0\le \int_{0}^{t} (t-r)^{(m-1)p^{*}} r^{\theta} \, dr\le  t^{(m-1)p^*}\int_{0}^{t}r^{\theta} \, dr=\frac{1}{\theta+1}t^{(m-1)p^*+\theta+1},
$$
for all $0<t<R$.
Noticing that $\gamma-p+1\ge \gamma-mp+1>0$ and $\theta\ge \gamma-mp>-1$ we can write
\begin{equation*}
\begin{aligned}
   & 0\le \bigg( \int_{0}^{t} (t-r)^{(m-1)p^{*}} r^{\theta} \, dr\bigg)^{\frac{1}{p^{*}}} \bigg( \int_{t}^{R} r^{-\frac{\gamma}{p-1}} \, dr\bigg)^{\frac{p-1}{p}}\\
   &\le \bigg( \frac{1}{\theta+1}t^{(m-1)p^*+\theta+1}\bigg)^{\frac{1}{p^{*}}}\bigg( \int_{t}^{R} r^{-\frac{\gamma}{p-1}} \, dr\bigg)^{\frac{p-1}{p}}\\
&=\left\{\begin{aligned}
&(\theta+1)^{-\frac{1}{p^{*}}}\left[\frac{p-1}{\gamma -p +1}\Big(1- \Big(\frac{t}{R}\Big)^{\frac{\gamma-p+1}{p-1}}\Big)\right]^{\frac{p-1}{p}},\;&\mbox{if}&\; 0<R<\infty\\
& (\theta+1)^{-\frac{1}{p^*}}\left[ \frac{p-1}{\gamma -p +1}\right]^{\frac{p-1}{p}},\;&\mbox{if}&\; R=\infty.
\end{aligned}\right.
\end{aligned}
\end{equation*}
where we have used that $m-1 + \frac{\theta+1}{p^{*}} - \frac{\gamma-p+1}{p}=0$. Thus, we obtain
\begin{eqnarray*}
    \mathcal{A}_{m, 0} = \sup_{0<t<R} \bigg( \int_{0}^{t} (t-r)^{(m-1)p^{*}} r^{\theta} \, dr\bigg)^{\frac{1}{p^{*}}} \bigg( \int_{t}^{R} r^{-\frac{\gamma}{p-1}} \, dr\bigg)^{\frac{p-1}{p}}<\infty,
\end{eqnarray*}
for all $0<R\le \infty$.
Analogously, $(r-t)^{p(m-1)/(p-1)}\le r^{p(m-1)/(p-1)}$ for any $0<t<r<R\le \infty$ yields
\begin{equation*}
\begin{aligned}
   & 0\le \bigg(\int_{0}^{t}  r^{\theta} \, dr\bigg)^{\frac{1}{p^{*}}}\bigg( \int_{t}^{R} (r-t)^{\frac{p(m-1)}{p-1}} r^{-\frac{\gamma}{p-1}} \, dr\bigg)^{\frac{p-1}{p}}\\
   &\le \left[ (\theta+1)^{-1} t^{\theta+1}\right]^{\frac{1}{p^{*}}}\left[\frac{p-1}{\gamma-mp+1}\Big( t^{- \frac{\gamma-mp+1}{p-1}}-R^{- \frac{\gamma-mp+1}{p-1}}\Big)\right]^{\frac{p-1}{p}} \\
&=\left\{\begin{aligned}
&(\theta+1)^{-\frac{1}{p^{*}}}\left[\frac{p-1}{\gamma -mp +1}\Big(1- \Big(\frac{t}{R}\Big)^{\frac{\gamma-mp+1}{p-1}}\Big)\right]^{\frac{p-1}{p}},\;&\mbox{if}&\; 0<R<\infty\\
& (\theta+1)^{-\frac{1}{p^*}}\left[ \frac{p-1}{\gamma -mp +1}\right]^{\frac{p-1}{p}},\;&\mbox{if}&\; R=\infty
\end{aligned}\right.
\end{aligned}
\end{equation*}
where we have used that $\frac{\theta+1}{p^{*}} -\frac{\gamma-mp+1}{p}=0$. Consequently, 
\begin{eqnarray*}
    \mathcal{A}_{m, 1} = \sup_{t > 0} \bigg( \displaystyle\int_{0}^{t}  r^{\theta} \, dr\bigg)^{\frac{1}{p^{*}}}\bigg( \int_{t}^{\infty} (r-t)^{\frac{p(m-1)}{p-1}} r^{-\frac{\gamma}{p-1}} \, dr\bigg)^{\frac{p-1}{p}}<\infty,
\end{eqnarray*}
for all  $0<R\le \infty$. It follows that $\mathcal{A}=\max \{ \mathcal{A}_{m, 0}, \mathcal{A}_{m, 1}\} < \infty$, for any $0<R\le \infty$. This proves item $(a)$. Analogously, taking into account the Remark~\ref{Remar-left} we can prove item $(b)$. Indeed, by assumptions $\theta \leq \gamma-mp$ and  $\gamma- p +1 < 0$ follows $1<p\le p^{*}$ and $\theta<-1$. So, we can apply Theorem~\ref{Teorema hardy} with $z(r) = r^{\theta}$ and $v(r)= r^{\gamma}$. For any
 $0<t<r<R\le \infty$ we have $(r-t)^{(m-1)p^*}\le r^{(m-1)p^*}$. Consequently, since $\theta + (m-1)p^*+1 = {(\theta+1)(\gamma-p+1)}/{(\gamma-mp+1)}:=\chi_0<0$, we can write
 \begin{eqnarray*}
    0&\le& \int_{t}^{R} (r-t)^{(m-1)p^{*}} r^{\theta} \, dr\\
    &\le&  \int_{t}^{R}r^{\theta + (m-1)p^*} \, dr\\
    &=&-\frac{1}{\chi_0} \left\{\begin{aligned}
&t^{\chi_0}\Big(1- \Big(\frac{R}{t}\Big)^{\chi_0}\Big),\;&\mbox{if}&\; 0<R<\infty\\
& t^{\chi_0},\;&\mbox{if}&\; R=\infty.
\end{aligned}\right.
\end{eqnarray*}
Thus,
\begin{equation*}
\begin{aligned}
   & 0\le \bigg( \int_{t}^{R} (r-t)^{(m-1)p^{*}} r^{\theta} \, dr\bigg)^{\frac{1}{p^{*}}}\bigg( \int_{0}^{t} r^{-\frac{\gamma}{p-1}} \, dr\bigg)^{\frac{p-1}{p}}\\
   &\le \bigg( \int_{t}^{R} (r-t)^{(m-1)p^{*}} r^{\theta} \, dr\bigg)^{\frac{1}{p^{*}}} \bigg[ \frac{1-p}{\gamma-p+1} t^{-\frac{\gamma-p+1}{p-1}}\bigg]^{\frac{p-1}{p}}\\
&=\left\{\begin{aligned}
&\bigg(-\frac{1}{\chi_0}\bigg)^{-\frac{1}{p^{*}}}\left[\frac{1-p}{\gamma -p + 1}\Big(1- \Big(\frac{R}{t}\Big)^{\chi_0}\Big)\right]^{\frac{p-1}{p}},\;&\mbox{if}&\; 0<R<\infty\\
& \bigg(-\frac{1}{\chi_0}\bigg)^{-\frac{1}{p^*}}\left[ \frac{1-p}{\gamma -p +1}\right]^{\frac{p-1}{p}},\;&\mbox{if}&\; R=\infty
\end{aligned}\right.
\end{aligned}
\end{equation*}
since that
$$\frac{\chi_0}{p^{*}} = \frac{\gamma-p+1}{p}.$$ 
Therefore, for every $0<R\le \infty$ we see
\begin{eqnarray*}
    \tilde{\mathcal{A}}_{m, 0} = \sup_{0<t<R} \bigg( \int_{t}^{R} (r-t)^{(m-1)p^{*}} r^{\theta} \, dr\bigg)^{\frac{1}{p^{*}}} \bigg( \int_{0}^{t} r^{-\frac{\gamma}{p-1}} \, dr\bigg)^{\frac{p-1}{p}}<\infty.
\end{eqnarray*}
To estimate $\tilde{\mathcal{A}}_{m,1}$ we proceed similarly as before. Since $$(t-r)^{\frac{p(m-1)}{p-1}} \leq t^{\frac{p(m-1)}{p-1}},\;\mbox{for all}\; 0<r<t<R \leq \infty $$ it follows that 
$$0 \le \int_{0}^{t} (t-r)^{\frac{p(m-1)}{p-1}} r^{-\frac{\gamma}{p-1}} \, dr \le \frac{1-p}{\gamma- p+1} t^{- \frac{\gamma-mp+1}{p-1}}. $$
Consequently,
\begin{equation*}
\begin{aligned}
   & 0\le \bigg( \int_{0}^{t} (t-r)^{\frac{(m-1)p}{p-1}} r^{-\frac{\gamma}{p-1}} \, dr\bigg)^{\frac{p-1}{p}} \bigg( \int_{t}^{R} r^{\theta} \, dr\bigg)^{\frac{1}{p^{*}}}\\
   &\le \bigg( \frac{1-p}{\gamma- p+1}  \bigg)^{\frac{p-1}{p}} t^{- \frac{\gamma-mp+1}{p}}  \bigg( \int_{t}^{R} r^{\theta} \, dr\bigg)^{\frac{1}{p^{*}}}\\
&=\left\{\begin{aligned}
&\bigg( \frac{1-p}{\gamma- p+1}  \bigg)^{\frac{p-1}{p}} [-(1+\theta)]^{-\frac{1}{p^{*}}} \Big(1- \Big(\frac{R}{t}\Big)^{\theta+1}\Big)^{\frac{1}{p^{*}}},\;&\mbox{if}&\; 0<R<\infty\\
& \bigg( \frac{1-p}{\gamma- p+1}  \bigg)^{\frac{p-1}{p}} [-(1+\theta)]^{-\frac{1}{p^{*}}},\;&\mbox{if}&\; R=\infty,
\end{aligned}\right.
\end{aligned}
\end{equation*}
since that
$$\dfrac{\theta+1}{p^{*}} = \dfrac{\gamma-mp+1}{p}.$$
Therefore, for every $0<R \leq \infty$ we see
$$\tilde{\mathcal{A}}_{m,1} = \sup_{0<t<R} \bigg( \int_{0}^{t} (t-r)^{\frac{(m-1)p}{p-1}} r^{-\frac{\gamma}{p-1}} \, dr\bigg)^{\frac{p-1}{p}} \bigg( \int_{t}^{R} r^{\theta} \, dr\bigg)^{\frac{1}{p^{*}}} < \infty.$$

\end{proof}
\begin{corollary}\label{D-embeding} Let $p>1$, $\theta, \alpha> -1$, $0<R\le \infty$ and $m\in \mathbb{Z}^{*}_{+}$ such that $\theta\ge \alpha-mp$ and $\alpha-mp+1>0$. Then, there is $C>0$ such that 
\begin{equation}\nonumber
     \left( \int_{0}^{R} |u(r)|^{p^{*}} r^{\theta} \, dr\right)^{\frac{1}{p^{*}}} \leq C \left( \int_{0}^{R} |u^{(m)}(r)|^{p} r^{\alpha} \, dr \right)^{\frac{1}{p}}, \;\; \mbox{with}\;\; p^{*}= \frac{(\theta +1)p}{\alpha -mp +1}
\end{equation}
holds, for any $u\in \mathcal{D}^{m,p}_{R}(\alpha)$. In particular, we have the continuous embedding 
\begin{equation}\label{D-emb}
    \mathcal{D}^{m,p}_{R}(\alpha)\hookrightarrow L^{p^{*}}_{\theta}.
\end{equation} 
\end{corollary}
\begin{proof}
Follows from Proposition~\ref{c0-prop1}, item $(a)$.
\end{proof}
\begin{corollary}\label{DX} Let $p>1$, $\alpha> -1$, $0<R\le \infty$ and $m\in \mathbb{Z}^{*}_{+}$ such that $\alpha-mp+1>0$. Then for any $0\le \ell\le m$ there is $C>0$ such that 
\begin{equation}\label{transition-jump}
      \begin{aligned}
     \left(\int_{0}^{R} |u^{(\ell)}(r)|^{p} r^{\alpha-(m-\ell)p} \, dr\right)^{\frac{1}{p}} & \le C \left( \int_{0}^{R} |u^{(m)}(r)|^{p} r^{\alpha} \, dr \right)^{\frac{1}{p}},
     \end{aligned}
\end{equation}
holds, for all $u \in AC^{m-1}_{\mathrm{R}}(0, R)$. In particular, we have
\begin{equation}\label{DX-Identity}
    \mathcal{D}^{m,p}_{R}(\alpha)=X^{m,p}_{R}(\beta_0,\cdots, \beta_{m}),
\end{equation}
 for the specific choices $\beta_{\ell}=\alpha-(m-\ell)p$. 
\end{corollary}
\begin{proof}
 Of course, we can assume $0\le \ell<m$.  If $u \in AC^{m-1}_{\mathrm{R}}(0, R)$ then $u^{(\ell)} \in AC^{m-\ell-1}_{\mathrm{R}}(0, R)$.
 In addition, we have $\alpha-(m-\ell)p+1\ge \alpha-mp+1>0$. So,  by choosing $\theta=\alpha-(m-\ell)p$ and $\gamma=\alpha$ we can apply the \eqref{c0-eprop1} for the space $AC^{m-\ell-1}_{\mathrm{R}}(0, R)$ to obtain 
      \begin{equation}\nonumber
      \begin{aligned}
     \left( \int_{0}^{R} |u^{(\ell)}(r)|^{p} r^{\alpha-(m-\ell)p} \, dr\right)^{\frac{1}{p}} & \leq C \left( \int_{0}^{R} |(u^{(\ell)})^{(m-\ell)}(r)|^{p} r^{\alpha} \, dr \right)^{\frac{1}{p}}\\
     &=C \left( \int_{0}^{R} |u^{(m)}(r)|^{p} r^{\alpha} \, dr \right)^{\frac{1}{p}}.
     \end{aligned}
\end{equation}
Thus proves \eqref{transition-jump}. Further, for the specific choices of weights $(\beta_0,\cdots, \beta_m)$, with $\beta_{\ell}=\alpha-(m-\ell)p$, \eqref{transition-jump} ensures  
$ D_{0,R}(\alpha)\subset W_{0,R}(\beta_0,\cdots, \beta_m)$ for any $0<R\le \infty.$
Since the contrary inclusion is obvious we get 
$$ D_{0,R}(\alpha)=W_{0,R}(\beta_0,\cdots, \beta_m),\;\; \mbox{for any}\;\; 0<R\le \infty.$$
In addition, from \eqref{transition-jump} we can see that the 
the full norm
\begin{equation}\nonumber
    \|u\|_{W^{m, p}_{R}} = \left( \sum_{\ell=0}^{m} \| u^{(\ell)}\|_{L^{p}_{\beta_{\ell}}}^{p} \right)^{\frac{1}{p}}
\end{equation}
is equivalent to the norm
\begin{equation}\nonumber
  \|u^{(m)}\|_{L^{p}_{\alpha}} = \left(\int_{0}^{R} |u^{(m)}(r)|^{p} r^{\alpha} \, dr \right)^{\frac{1}{p}}.
\end{equation}
 It follows that  for any $p>1$, $m\in \mathbb{Z}^{*}_{+}$ and $\alpha>-1$ such that $\alpha-mp+1>0$ we have
 $$\mathcal{D}^{m,p}_{R}(\alpha)=X^{m,p}_{R}(\beta_0,\cdots, \beta_{m}),\;\; \mbox{for any}\;\; 0<R\le \infty,$$
 with $\beta_{\ell}=\alpha-(m-\ell)p$.
 \end{proof}
\begin{remark} For $0<R<\infty$, the norms \eqref{c0-norma1} and \eqref{norma-diric} are equivalent on $W_{0,R}(\alpha_{0}, \cdots, \alpha_{m})$ for general class of weights  $(\alpha_0,\cdots, \alpha_{m})$ satisfying the transition condition \eqref{trasition-C}, see \cite{JAJ} for more details. Thus, the identity 
    $\mathcal{D}^{m,p}_{R}(\alpha_m)=X^{m,p}_{R}(\alpha_0,\cdots, \alpha_{m})$ can be obtained for more general class of weights. However, Corollary~\ref{DX}  has the advantage of including $R=\infty$.
\end{remark}  
In the next we extend the extension type operator obtained in \cite[Lemma~2.1]{JFO} for higher order derivatives $m\ge 2$.
\begin{lemma}[Extension operator]\label{c0-lema de extensao}
Let $R>0$ and $p\ge 1$ be real numbers and $m \in \mathbb{Z}^{*}_{+}$. For each $2R<L\le \infty$ there exists a linear extension operator $T: W^{m,p}_{R} \to X^{m,p}_{L}$
such that
\begin{itemize}
    \item [$(a)$] $Tu = u$ in $(0,R)$,
    \item [$(b)$] \textrm{supp} $Tu \subseteq [0, 2R)$,
    \item [$(c)$]     $\|Tu\|_{W^{m,p}_{L}} \leq C \|u\|_{W^{m,p}_{R}}$, for some $C = C(m, p, R, \alpha_{0}, \cdots, \alpha_{m}) > 0$.
\end{itemize}
\end{lemma}
\begin{proof}
Let $\eta \in C^{\infty}_{c}[0, \infty)$ be an auxiliary function satisfying
\begin{equation*}
\eta(r) = \left\{\begin{aligned}
&1, \;\; &\mbox{if}&\;\; 0 \leq r \leq \frac{R}{4}\\
& 0 \leq \eta(r) \leq 1, \,\,&\mbox{if}&\,\, \frac{R}{4} \leq r \leq \frac{3R}{4}\\
&0, \;\; &\mbox{if}&\;\; r> \frac{3R}{4}.
\end{aligned}\right.
\end{equation*}
For $u \in W^{m,p}_{R}$, we define $v_{1}, v_{2} : [0, L]  \to \mathbb{R}$ by
\begin{equation*}
    v_{1}(r) = \left\{\begin{aligned}
     &  \eta(r) u(r), \;\; &\mbox{if}&\;\; 0 \leq r \leq R\\
& 0, \;\; &\mbox{if}&\;\; R < r \leq L
    \end{aligned}\right.
\end{equation*}
and
\begin{equation*}
    v_{2}(r) =\left\{\begin{aligned}
&\big(1-\eta(r) \big) u(r),\;\; &\mbox{if}&\;\; 0\leq r \leq R\\
&\big(1- \eta(2R-r)\big)u(2R-r),\;\; &\mbox{if}&\;\; R < r \leq \frac{7R}{4}\\
&0, \;\; &\mbox{if}&\;\;\frac{7R}{4} \leq r \leq L.
\end{aligned}\right.
\end{equation*}
For any $l= 0, 1, \cdots, m$ holds 
\begin{equation*}
    v^{(l)}_{1}(r) = \left\{\begin{aligned}
&\sum^{l}_{j=0} {l\choose j}\eta^{(j)}(r) u^{(l-j)}(r), \;\; &\mbox{if}&\;\; 0 < r \leq R,\\
&0, \;\;&\mbox{if}&\;\; R < r < L
\end{aligned}\right.
\end{equation*}
and 
\begin{equation*}
    v^{(l)}_{2}(r) =\left\{\begin{aligned}
& \big(1-\eta(r) \big) u^{(l)}(r) - \displaystyle\sum^{l}_{j=1} {l\choose j} \eta^{(j)}(r) u^{(l-j)}(r),\;\; &\mbox{if} &\;\; 0< r \leq R\\
& \left. \begin{aligned}
&\big(1-\eta(2R - r) \big) u^{(l)}(2R-r) \\
&- \displaystyle\sum^{l}_{j=1} {l\choose j} \eta^{(j)}(2R-r) u^{(l-j)}(2R-r)\end{aligned}\right\}, \;\; &\mbox{if}&\;\; R < r \leq\frac{7R}{4},\\
& 0,\;\; &\mbox{if}&\;\;\frac{7R}{4} \leq r < L.
\end{aligned}\right.
\end{equation*}
Hence, $v_{1}^{(m-1)}$, $v_{2}^{(m-1)} \in AC_{loc}[0, L]$, with $v_{1}^{(l)}(L) = v_{2}^{(l)}(L) = 0$, for $l= 0,1, \cdots, m-1$. Then, we define $T: W^{m,p}_{R} \to X^{m,p}_{L}$ by setting $$Tu = v_{1} + v_{2}.$$ It is easy to verify that $T$ is a linear operator with
$Tu = u$ on $(0,R)$ and  $\textrm{supp} \,Tu \subseteq [0, 2R)$ for any $u\in  W^{m,p}_{R}$. We proceed to show that item $(c)$ holds. Of course, 
\begin{equation}\label{est0}
    \|Tu\|_{W^{m,p}_{L}} \leq \|v_{1}\|_{W^{m,p}_{L}} + \|v_{2}\|_{W^{m,p}_{L}}.
\end{equation}
We claim that there exists $C=C(\alpha_0, \cdots, \alpha_m, m, p, R)>0$ such that
\begin{equation}\label{est+}
    \max\{ \|v_{1}\|_{W^{m,p}_{L}},  \|v_{1}\|_{W^{m,p}_{L}}\} \leq C \|u\|_{W^{m,p}_{R}}.
\end{equation}
In fact, for any $l = 0, 1, \cdots, m$, we have  
$$|v_{1}^{(l)}|^{p} \leq 2^{lp} \displaystyle\sum^{l}_{j=0} {l\choose j}|\eta^{(j)}(r)|^{p} |u^{(l-j)}(r)|^{p} \quad \text{on} \quad (0,R).$$ 
Then, 
\begin{eqnarray}\label{est1}
\|v_{1}^{(l)}\|^{p}_{L^{p}_{\alpha_{l}}} &\leq &   2^{lp} \int_{0}^{R} \bigg[ \displaystyle\sum^{l}_{j=0} {l\choose j}  |\eta^{(j)}(r)|^{p} |u^{(l-j)}(r)|^{p} r^{\alpha_{l}} \bigg] \, dr  \nonumber \\ 
&=& 2^{lp}  \bigg[ \displaystyle\sum^{l}_{j=1} {l\choose j}  \int_{\frac{R}{4}}^{R} r^{\alpha_{l}- \alpha_{l-j}} |\eta^{(j)}(r)|^{p} |u^{(l-j)}(r)|^{p} r^{\alpha_{l-j}} \, dr  \nonumber \\ 
&+&   \int_{0}^{R}| \eta(r)|^{p} |u^{(l)}(r)|^{p} r^{\alpha_{l}} \, dr \bigg].
\end{eqnarray}
For any $j = 1, 2, \cdots, l$, application $r \mapsto r^{\alpha_{l}- \alpha_{l-j}} |\eta^{(j)}(r)|^{p}$ is bounded in
 $[R/4, R]$. Therefore, combining this fact with
 \eqref{est1}, we have
\begin{equation}\label{est2}
\|v_{1}^{(l)}\|^{p}_{L^{p}_{\alpha_{l}}} \leq C_{l} \| u \|^{p}_{W^{m,p}_{R}},    
\end{equation}
for some constant $C_{l}= C_{l}(l, p, R, \alpha_0, \cdots, \alpha_l) >0$. Using  \eqref{est2} and we choose $C =\max\{C_0,\cdots, C_m\}$, we have
\begin{equation}\label{est3}
\|v_{1}\|^{p}_{W^{m,p}_{L}} \leq C \| u \|^{p}_{W^{m,p}_{R}}.    
\end{equation}
Similarly, for any $l=0, 1, \cdots, m$, holds 
$$|v_{2}^{(l)}|^{p} \leq 2^{lp} \bigg[ |1-\eta(r)| |u^{(l)}(r)| + \sum^{l}_{j=1} {l\choose j} |\eta^{(j)}(r)| |u^{(l-j)}(r)| \bigg] \quad \text{on} \quad (0,R)$$
and
\begin{equation}
\label{estimativa1}
\begin{aligned}
    |v_{2}^{(l)}|^{p} &\le 2^{lp} \bigg[ |1-\eta(2R-r)| |u^{(l)}(2R-r)| \\
&+ \sum^{l}_{j=1} {l\choose j} |\eta^{(j)}(2R-r)| |u^{(l-j)}(2R-r)| \bigg] \quad \text{on} \;\; (R, 7R/4).
\end{aligned}
\end{equation}
Moreover, we note that the applications 
$$r \mapsto (2R-r)^{\alpha_{l}} r^{-\alpha_{l}} \quad \text{and} \quad r \mapsto |\eta^{(j)}(r)|^{p} (2R-r)^{\alpha_{l}} r^{-\alpha_{l-j}}$$
are bounded in $[R/4, R]$, for $j = 1, 2, \cdots, l$.
Therefore, combining this fact with
 \eqref{estimativa1}, we have
\begin{equation}\label{est4}
\|v_{2}\|^{p}_{W^{m,p}_{L}} \leq C \| u \|^{p}_{W^{m,p}_{R}}.    
\end{equation}
Finally, combining \eqref{est3} and \eqref{est4}, we have \eqref{est+}.  
By \eqref{est+} and \eqref{est0}, we complete the proof of item (c).
\end{proof}

By combining the Lemma \ref{c0-lema de extensao} with \cite[Lemma 2.2]{JAJ}, we are able to give the behavior of any $u \in \mathcal{D}^{m,p}_{R}(\alpha)$ with $0<R\le \infty$ close to the origin.
\begin{lemma}\label{c0-corolario4}
Let $u \in \mathcal{D}^{m,p}_{R}(\alpha)$ with $0<R\le \infty$ and $\alpha-mp+1>0$. For
 $\theta>0$ such that $p \theta  \geq \alpha - p +1$, we have
$$\lim_{r \to 0} r^{\theta} u^{(j)}(r)= 0, \quad \forall \, j= 0,1, \cdots, m-1.$$ 
\end{lemma}
\begin{proof} From Corollary~\ref{DX}, for $\beta_{\ell}=\alpha-(m-\ell)p$ we have 
$$\mathcal{D}^{m,p}_{R}(\alpha)=X^{m,p}_{R}(\beta_0,\cdots, \beta_{m}),\;\mbox{for}\; 0<R\le \infty. $$
 Then, if  $0< R< \infty$,  the result follows directly from \cite[Lemma 2.2]{JAJ}. Now, if $u \in \mathcal{D}^{m,p}_{\infty}(\alpha)=X^{m,p}_{\infty}(\beta_0,\cdots, \beta_{m})$ we necessarily have that $u \in W^{m,p}_{S}(\beta_0,\cdots, \beta_{m})$, for any $0<S<\infty$. By Lemma \ref{c0-lema de extensao}, there is a linear extension operator $$T: W^{m,p}_{S}(\beta_0,\cdots, \beta_{m}) \to X^{m,p}_{L}(\beta_0,\cdots, \beta_{m})$$ such that
 $Tu = u$ on $(0,S)$, provided that $2S<L<\infty$. 
Since $Tu \in X^{m,p}_{L}(\beta_0,\cdots, \beta_{m})$ we are in a position to apply \cite[Lemma 2.2]{JAJ} to obtain
$$\lim_{r \to 0} r^{\theta} u^{(j)}(r)=\lim_{r \to 0} r^{\theta} (Tu)^{(j)}(r)= 0, \quad \text{for}\,\,  j= 0,1, \cdots, m-1.$$
\end{proof}
Lemma~\ref{c0-lema3} and Lemma~\ref{c0-lema4} below have been proven in \cite[Lemma~4.2, Lemma~4.3]{JAJ} for $0<R<\infty$ and the argument from there works for $R=\infty$. So,  we omit the proof here.
\begin{lemma}\cite[Lemma~4.2]{JAJ}\label{c0-lema3}
Let $\alpha \in \mathbb{R}_{+}^{*}$, $0< R\le \infty$ and $k \in \mathbb{Z}^{*}_{+}$. There are constants $c_{1}, \cdots, c_{2k-1}$  and $d_{1}, \cdots, d_{2k}$ such that
\begin{equation}\label{c0-e44}
    \Delta^{k}_{\alpha} u = u^{(2k)} + \sum_{i=1}^{2k-1} c_{i} \dfrac{u^{(2k-i)}}{r^{i}},\;\text{for all}\; u \in AC^{2k-1}_{loc}(0,R)
\end{equation}
and 
\begin{equation}\label{c0-e45}
    [\Delta^{k}_{\alpha} u]' = u^{(2k+1)} + \sum_{i=1}^{2k} d_{i} \dfrac{u^{(2k+1-i)}}{r^{i}}, \;\text{for all}\;u \in AC^{2k}_{loc}(0,R).
\end{equation}
\end{lemma}
\begin{lemma}\cite[Lemma~4.3]{JAJ}\label{c0-lema4}
Let $\alpha \in \mathbb{R}_{+}^{*}$ and $0< R\le \infty$. There are constants $c_{1}, \cdots, c_{2k-1}$ and $d_{1}, \cdots, d_{2k}$  such that
\begin{equation}\label{c0-e49}
    \Delta^{k}_{\alpha} u = u^{(2k)} + c_{1} \dfrac{(\Delta_{\alpha}^{k-1} u)'}{r} + c_{2} \dfrac{\Delta_{\alpha}^{k-1} u}{r^{2}} + \cdots + c_{2k-1} \dfrac{(\Delta_{\alpha}^{0} u)'}{r^{2k-1}},
\end{equation}
for all $u \in AC^{2k-1}_{loc}(0,R)$ with $k \in \mathbb{Z}^{*}_{+}$, and
\begin{equation}\label{c0-e50}
    [\Delta^{k}_{\alpha} u]' = u^{(2k+1)} + d_{1} \dfrac{\Delta_{\alpha}^{k} u}{r} + d_{2} \dfrac{(\Delta_{\alpha}^{k-1} u)'}{r^{2}} + d_{3} \dfrac{\Delta_{\alpha}^{k-1} u }{r^{3}} +\cdots + d_{2k} \dfrac{(\Delta_{\alpha}^{0} u)'}{r^{2k}}, 
\end{equation}
for all $u \in AC^{2k}_{loc}(0, R)$ with $k \in \mathbb{Z}_{+}$.
\end{lemma}
\begin{corollary}\label{Navier-boundary}
For $0<R\le \infty$,  $m\in\mathbb{Z}^{*}_{+}$ and $\alpha>0$ we have
\begin{equation}\label{c0-e56}
   \lim_{r \to R} (\Delta^{l}_{\alpha} u)(r) = 0\;\;\mbox{if} \;\; 0\le l<m/2
\end{equation}
for any $u\in AC^{m-1}_{\mathrm{R}}(0,R)$.
 \end{corollary}
 \begin{proof} Of course the result holds for $l=0$. Let $1\le l<m/2$ be an integer number. If $m=2k$ is even then $u\in AC^{m-1}_{\mathrm{R}}(0,R)$ it means that $u\in AC^{2k-1}_{loc}(0,R)$ and $\lim_{r \to R} u^{(j)}(r) = 0$ for all $0\le j\le 2k-1$. Hence, $u\in AC^{2l-1}_{loc}(0,R)$ and $\lim_{r \to R} u^{(j)}(r) = 0$ for all $0\le j\le 2l\le 2k-1$. So, the expansion \eqref{c0-e44} yields \eqref{c0-e56}. Analogously, for $m=2k+1$ odd, $u\in AC^{m-1}_{\mathrm{R}}(0,R)$ it means that $u\in AC^{2k}_{loc}(0,R)$ and $\lim_{r \to R} u^{(j)}(r) = 0$ for all $0\le j\le 2k$ which implies  $u\in AC^{2l-1}_{loc}(0,R)$ and $\lim_{r \to R} u^{(j)}(r) = 0$ for all $0\le j\le 2l\le 2k$ and the result follows from \eqref{c0-e44} again.
 \end{proof}

\begin{corollary} Let $u\in D_{0,R}(\alpha,m,p)$, $p\ge 2$, $0<R\le \infty$ and $\alpha-mp+1>0$. Then
\begin{equation}\label{c0-e57}
    \lim_{r \to 0} r^{\alpha} (\Delta_{\alpha}^{l} u)(r) = 0\,\,\, \text{and}\,\,\, \lim_{r \to 0} r^{\alpha}(\Delta_{\alpha}^{l} u)'(r) = 0,
\end{equation}
for all integer number  $0\le l<m/2$. 
\end{corollary}
\begin{proof}
   Since $\alpha p\ge \alpha-p+1$,  Corollary~\ref{c0-corolario4} yields \eqref{c0-e57} for $l=0$. Let us take $l\in\mathbb{N}$ such that $1\le l<m/2$. Firstly, $u\in D_{0,R}(\alpha,m,p)$ implies  $u\in AC^{m-1}_{\mathrm{R}}(0,R)$  and, thus $u\in AC^{2l}_{\mathrm{R}}(0,R)$ because $2l\le m-1$.  Thus,  from the expansion \eqref{c0-e44}, in order to ensure the first limit in \eqref{c0-e57} is sufficient to show that 
\begin{equation}\label{c0-57-equiv}
    \lim_{r\to 0}r^{\alpha-i}u^{(2l-i)}(r)=0,  \,\, \mbox{for all} \,\, i=0,1, \cdots, 2l-1.
\end{equation}
To get \eqref{c0-57-equiv}, we will use  Corollary~\ref{c0-corolario4}.  Note that $2l<m$, $p>1$ and $\alpha- mp + 1 >0$ yields
\begin{align*}
    \theta_i:=\alpha-i& \ge\alpha-(2l-1)=\alpha-2l+1\\
    &\ge \alpha-m+1\ge \alpha-mp+1>0
\end{align*}
for all $i=0,1, \cdots, 2l-1.$
So, by using $\alpha-i\ge \alpha-m+1$, $mp<\alpha +1$ and  $p\ge 2$
\begin{equation}\label{Ai}
\begin{aligned}
   \theta_i p=(\alpha-i)p & \ge(\alpha-m+1)p\\
    &=(\alpha+1)p-mp\\
    &>(\alpha+1)p-(\alpha+1)\\
    &=(\alpha+1)(p-1)\\
    &\ge \alpha+1>\alpha-p+1.
\end{aligned}
\end{equation}
Thus, we can apply Lemma \ref{c0-corolario4} with $\theta_i=\alpha-i>0$ to obtain \eqref{c0-57-equiv}. Analogously, from \eqref{c0-e45},  the second limit in \eqref{c0-e57} holds provided that 
\begin{equation}\label{c0-57-equiv-dev}
    \lim_{r\to 0}r^{\alpha-i}u^{(2l+1-i)}(r)=0,  \,\, \mbox{for all} \,\, i=0,1, \cdots, 2l.
\end{equation}
To prove \eqref{c0-57-equiv-dev}, we argue similar to \eqref{c0-57-equiv}. First, for $p>1$
\begin{align*}
    \alpha-i& \ge\alpha-2l\ge \alpha-m+1\ge \alpha-mp+1>0.
\end{align*}
Hence, arguing as in \eqref{Ai}, we also have 
\begin{equation}\nonumber
\begin{aligned}
    (\alpha-i)p >\alpha-p+1.
\end{aligned}
\end{equation}
By  applying Lemma \ref{c0-corolario4} again, we obtain \eqref{c0-57-equiv-dev}.
\end{proof}
\begin{corollary}\label{corollary right side}  Let $0<R\le \infty$, $p>1$, $\alpha>-1$ and $m\in \mathbb{Z}^{*}_{+}$ such that $\alpha-mp+1>0$. Then, 
\begin{equation}\label{right side}
    \|\nabla^{m}_{\alpha} u\|_{L^{p}_{\alpha}}\le C \|u^{(m)}\|_{L^{p}_{\alpha}}, \;\;\mbox{for any}\;\; u\in \mathcal{D}^{m,p}_{R}(\alpha)
\end{equation}
for some constant $C>0$.
	\end{corollary} 
	\begin{proof}
For any $u\in D_{0,R}(\alpha,m,p)$, from \eqref{transition-jump} there exists $C>0$ independent of $u$ such that  
 \begin{equation}\label{iteration up}
     \left\|\frac{u^{(\ell)}}{r^{m-\ell}}\right\|_{L^{p}_{\alpha}}\le C \left\|u^{(m)}\right\|_{L^{p}_{\alpha}}, \; 0\le \ell \le m.
 \end{equation}
 Hence, \eqref{right side} it follows from \eqref{c0-e44} and \eqref{iteration up} if $m$ is even and from  \eqref{c0-e45} and \eqref{iteration up} if $m$ is odd.
	\end{proof}
Next, we apply Proposition~\ref{c0-prop1} to extend \cite[Lemma~4.4]{JAJ} for the  $\mathcal{D}^{m,p}_{R}(\alpha)$ with $0<R\le \infty$. 

\begin{lemma}\label{c0-lema5} Let  $0<R\le \infty$ and $\alpha>0$. Suppose that $m\in\mathbb{Z}^{*
}_{+}$, $p\ge 2$ such that $\alpha-mp+1>0$. Then,
\begin{enumerate}
    \item [$(a)$] If $m$ is even, there is $c>0$ such that for any $1\le l\le m/2$ and  $u\in\mathcal{D}^{m,p}_{R}(\alpha)$
    \begin{itemize}
    \item [$(i)$] $\Bigg\|\dfrac{(\Delta_{\alpha}^{\frac{m}{2}-l} u)'}{r^{2l-1}} \Bigg\|_{L^{p}_{\alpha}} \leq c \|\nabla^{m}_{\alpha}u\|_{L_{\alpha}^{p}}$
    \item [$(ii)$] $\Bigg\|\dfrac{\Delta_{\alpha}^{\frac{m}{2}-l} u}{r^{2l}} \Bigg\|_{L^{p}_{\alpha}} \leq c  \|\nabla^{m}_{\alpha}u\|_{L_{\alpha}^{p}}$.
\end{itemize}
\item [$(b)$] If $m$ is odd, there is $c>0$ such that for any $1\le l\le (m-1)/2$ and $u\in\mathcal{D}^{m,p}_{R}(\alpha)$
\begin{itemize}
    \item [$(i)$] $\Bigg\|\dfrac{\Delta_{\alpha}^{\frac{m-1}{2}-(l-1)} u}{r^{2l-1}} \Bigg\|_{L^{p}_{\alpha}} \leq c \|\nabla^{m}_{\alpha}u\|_{L_{\alpha}^{p}}$
    \item [$(ii)$] $\Bigg\|\dfrac{(\Delta_{\alpha}^{\frac{m-1}{2}-l} u)'}{r^{2l}} \Bigg\|_{L^{p}_{\alpha}} \leq c \|\nabla^{m}_{\alpha}u\|_{L_{\alpha}^{p}}$.
\end{itemize}
\end{enumerate}
\end{lemma}
\begin{proof}
$(a)$-$(i)$ Assume $m=2k$ with $k\in\mathbb N$. First, we observe that $$\alpha-p\alpha-p+1=(\alpha+1)(1-p)<0.$$
Thus, since from \eqref{c0-e57} we have $r^{\alpha}(\Delta^{k-1}_{\alpha}u)^{\prime}\in AC^{0}_{\mathrm{L}}(0,R)$, we can apply the Proposition \ref{c0-prop1}-$(b)$ with $m=1$, $\theta=\alpha-p\alpha-p$ and  $\gamma=\theta+p$ to obtain
\begin{equation}\label{1passo}
\begin{aligned}
      \Bigg\|\dfrac{(\Delta_{\alpha}^{k-1} u)'}{r} \Bigg\|_{L^{p}_{\alpha}}^{p} &= \int_{0}^{R} |r^{\alpha} (\Delta_{\alpha}^{k-1} u)'|^{p} r^{\alpha-p\alpha-p} \, dr \\
&\leq C_1 \int_{0}^{R} |\big(r^{\alpha} (\Delta_{\alpha}^{k-1} u)'\big)'|^{p} r^{\alpha-p\alpha} \, dr \\
&= C_1 \int_{0}^{R} |\Delta_{\alpha}^{k} u|^{p} r^{\alpha} \, dr,
  \end{aligned}  
\end{equation}
which proves $(a)$-$(i)$ for $l=1$.
Suppose that $2\le l\le k$. For this case, we need at least two steps. In fact, for any
$j\in \big\{2, \cdots, k\big\}$ we also have $$\alpha-p\alpha-2pj+p+1=(\alpha+1)(1-p)-2p(j-1)<0.$$ 
From \eqref{c0-e57} we have $r^{\alpha}(\Delta^{k-j}_{\alpha}u)^{\prime}\in AC^{0}_{\mathrm{L}}(0,R)$. Then, we can apply the Proposition \ref{c0-prop1}-$(b)$ with $m=1$, $\theta=\alpha-p\alpha-2pj+p$ and $\gamma=\theta+p$ to obtain
\begin{equation}\label{c0-e59}
\begin{aligned}
      \Bigg\|\dfrac{(\Delta_{\alpha}^{k-j} u)'}{r^{2j-1}} \Bigg\|_{L^{p}_{\alpha}}^{p} &= \int_{0}^{R} |r^{\alpha} (\Delta_{\alpha}^{k-j} u)'|^{p} r^{\alpha-p\alpha-2pj+p} \, dr \\
&\leq c_1 \int_{0}^{R} |\big(r^{\alpha} (\Delta_{\alpha}^{k-j} u)'\big)'|^{p} r^{\alpha-p\alpha-2pj+2p} \, dr \\
&= c_1 \int_{0}^{R} |\Delta_{\alpha}^{k-(j-1)} u|^{p} r^{\alpha - 2p(j-1)} \, dr.
  \end{aligned} 
\end{equation}
Also, since $\alpha-2kp+1>0$ we have
$
\alpha - 2p(j-1)+1=\alpha-2pj+1+2p>0.
$
In addition, from Corollary~\ref{Navier-boundary} we have $\Delta_{\alpha}^{k-(j-1)} u\in AC^{0}_{\mathrm{R}}(0,R)$, $j\ge 2$. Then, by using Proposition \ref{c0-prop1}-$(a)$ with $m=1$, $\theta=\gamma - 2p(j-1)$ and $\gamma=\theta+p$ we get 
\begin{equation}\label{c0-e60}
\int_{0}^{R} |\Delta_{\alpha}^{k-(j-1)} u|^{p} r^{\alpha - 2p(j-1)} \, dr \le c_2 \int_{0}^{R} |(\Delta_{\alpha}^{k-(j-1)} u)'|^{p} r^{\alpha - 2p(j-1)+p} \, dr.
\end{equation}
By combining  \eqref{c0-e59} with \eqref{c0-e60},  we have positive constant $C=C_j>0$ such that 
\begin{equation}\label{2passos}
    \begin{aligned}
        \Bigg\|\dfrac{(\Delta_{\alpha}^{k-j} u)'}{r^{2j-1}} \Bigg\|_{L^{p}_{\alpha}}^{p}  & \leq 
C_j\Bigg\|\dfrac{(\Delta_{\alpha}^{k-(j-1)} u)'}{r^{2(j-1)-1}} \Bigg\|_{L^{p}_{\alpha}}^{p},
    \end{aligned}
\end{equation}
for any $j\in \big\{2, \cdots, k\big\}$. For any $2\le l\le k$, by successively using the estimate in \eqref{2passos}, we get 
\begin{equation}\label{cadeia}
    \begin{aligned}
        \Bigg\|\dfrac{(\Delta_{\alpha}^{k-l} u)'}{r^{2l-1}} \Bigg\|_{L^{p}_{\alpha}}^{p}  & \leq 
C_l\Bigg\|\dfrac{(\Delta_{\alpha}^{k-(l-1)} u)'}{r^{2(l-1)-1}} \Bigg\|_{L^{p}_{\alpha}}^{p}\\
& \leq 
C_{l}C_{l-1}\Bigg\|\dfrac{(\Delta_{\alpha}^{k-(l-2)} u)'}{r^{2(l-2)-1}} \Bigg\|_{L^{p}_{\alpha}}^{p}\\
& \leq 
C_{l}C_{l-1}C_{l-2}\Bigg\|\dfrac{(\Delta_{\alpha}^{k-(l-3)} u)'}{r^{2(l-3)-1}}\Bigg\|_{L^{p}_{\alpha}}^{p}\\
&\le C_{l}C_{l-1}C_{l-2}\cdots C_4C_{3}C_{2}\Bigg\|\dfrac{(\Delta_{\alpha}^{k-1} u)'}{r}\Bigg\|_{L^{p}_{\alpha}}^{p}.
    \end{aligned}
\end{equation}
Then,  \eqref{1passo} and \eqref{cadeia} yield $(a)$-$(i)$. 

To prove the item $(a)$-$(ii)$, we proceed by similar argument. Here, for any $j\in\big\{1,2, \cdots, k\big\}$ we have $\alpha-2pj+1\ge \alpha-2pk+1>0$ and, from \eqref{c0-e56},  $\Delta^{k-j}_{\alpha}u\in AC^{0}_{\mathrm{R}}(0,R)$. Thus, Proposition \ref{c0-prop1}-$(a)$ with $m=1$, $\theta=\alpha - 2pj$ and $\gamma=\theta+p$ implies
\begin{equation}\nonumber
 \begin{aligned}
    \Bigg\|\dfrac{\Delta_{\alpha}^{k-j} u}{r^{2j}} \Bigg\|^p_{L^{p}_{\alpha}}& =\int_{0}^{R}|\Delta_{\alpha}^{k-j} u|^{p}r^{\alpha-2pj}dr\le c_1 \int_{0}^{R}|(\Delta_{\alpha}^{k-j} u)^{\prime}|^{p}r^{\alpha-2pj+p}dr\\
    &= c_1 \int_{0}^{R}|r^{\alpha}(\Delta_{\alpha}^{k-j} u)^{\prime}|^{p}r^{\alpha-p\alpha-2pj+p}dr.
\end{aligned}   
\end{equation}
In addition,  we have $\alpha-p\alpha-2pj+p+1=(\alpha+1)(1-p)-2p(j-1)<0$ and, from \eqref{c0-e57}, $r^{\alpha}(\Delta_{\alpha}^{k-j} u)^{\prime}\in AC^{0}_{\mathrm{L}}(0,R)$. Thus, Proposition \ref{c0-prop1}-$(b)$ with $m=1$, $\theta=\alpha-p\alpha-2pj+p$ and $\gamma=\theta+p$ together with the above inequality ensures the existence of $C=C_j>0$ such that 
\begin{equation}\nonumber
 \begin{aligned}
    \Bigg\|\dfrac{\Delta_{\alpha}^{k-j} u}{r^{2j}} \Bigg\|^p_{L^{p}_{\alpha}}& \le 
    C_j\Bigg\|\dfrac{\Delta_{\alpha}^{k-(j-1)} u}{r^{2(j-1)}} \Bigg\|^p_{L^{p}_{\alpha}}.
\end{aligned}   
\end{equation}
By iterating the above estimate  we can see that $(a)$-$(ii)$ holds for any $1\le l\le k$. By similarly argument $(b)$ holds.
\end{proof}
\begin{proposition}\label{Equi-norms}  Let  $0<R\le \infty$ and $\alpha>0$. Suppose that $m\in\mathbb{Z}^{*
}_{+}$, $p\ge 2$ such that $\alpha-mp+1>0$. Then, there are $C_1,C_2>0$ such that
    \begin{equation}
        C_1\|u^{(m)}\|_{L^{p}_{\alpha}}\le \|\nabla^{m}_{\alpha}u\|_{L^{p}_{\alpha}}\le C_2\|u^{(m)}\|_{L^{p}_{\alpha}}
    \end{equation}
    for all $u\in \mathcal{D}^{m,p}_{R}(\alpha)$.
\end{proposition}
\begin{proof} The existence of $C_2>0$ is ensured by Corollary~\ref{corollary right side}. To get $C_1>0$, we combine Lemma~\ref{c0-lema5} with the expansions in \eqref{c0-e49} and \eqref{c0-e50}. In fact, if  $m=2k$ is an even integer number 
From \eqref{c0-e49}, we have
\begin{eqnarray*}
\|u^{(m)} \|_{L^{p}_{\alpha}} &\leq &  \|\Delta^{k}_{\alpha} u \|_{L^{p}_{\alpha}}+ c_{1} \Bigg\|\dfrac{(\Delta_{\alpha}^{k -1} u)'}{r} \Bigg\|_{L^{p}_{\alpha}}+ c_{2} \Bigg\|\dfrac{\Delta_{\alpha}^{k-1} u}{r^{2}} \Bigg\|_{L^{p}_{\alpha}} + \cdots \nonumber\\
&+& c_{2k-3} \Bigg\|\dfrac{(\Delta_{\alpha}u)'}{r^{2k-3}} \Bigg\|_{L^{p}_{\alpha}} + c_{2k-2} \Bigg\|\dfrac{\Delta_{\alpha} u}{r^{2k-2}} \Bigg\|_{L^{p}_{\alpha}} + c_{2k-1} \Bigg\|\dfrac{(\Delta_{\alpha}^{0} u)'}{r^{2k-1}} \Bigg\|_{L^{p}_{\alpha}}.
\end{eqnarray*}
Thus,  Lemma~\ref{c0-lema5}, item $(a)$ ensures the existence of $C_1>0$. Analogously, if  $m=2k+1$ is odd, from \eqref{c0-e50}
\begin{eqnarray*}
\|u^{(m)} \|_{L^{p}_{\alpha}} &\leq&  \|(\Delta^{k}_{\alpha} u)' \|_{L^{p}_{\alpha}}+ d_{1} \Bigg\| \dfrac{\Delta_{\alpha}^{k} u}{r} \Bigg\|_{L^{p}_{\alpha}}+ d_{2} \Bigg\|\dfrac{(\Delta_{\alpha}^{k-1} u)'}{r^{2}} \Bigg\|_{L^{p}_{\alpha}} + \cdots \nonumber\\
&+& d_{2k-1} \Bigg\|\dfrac{\Delta_{\alpha} u}{r^{2k-1}} \Bigg\|_{L^{p}_{\alpha}} + d_{2k} \Bigg\|\dfrac{(\Delta_{\alpha}^{0} u)'}{r^{2k}} \Bigg\|_{L^{p}_{\alpha}}.
\end{eqnarray*}
and the existence of $C_1$ is ensured by Lemma~\ref{c0-lema5}, item $(b)$.
\end{proof}
\begin{proof}[Proof of Theorem \ref{c0-t2}] Follows directly from Corollary~\ref{D-embeding} and Proposition~\ref{Equi-norms}.
\end{proof}
\section{Attainability of the best Sobolev-type constant}
\label{sec3}
In this section we will present the proof of Theorem \ref{a0-teo1} which requires a concentration-compactness type result in the same line of the classical one due to P.L. Lions \cite{PL3}  and \cite{PL4}.

We start to show the following  useful property of the constant $\mathcal{S}(m,p,\alpha, \theta, R)$, see \cite{Clement-deFigueiredo-Mitidieri} for the case $m=1$.

Throughout this section we are assuming  $m$, $\alpha, \theta, R$ and $p$ under the assumptions of Theorem~\ref{c0-t2}. 

\begin{lemma}\label{remark dilation invariance}
Let $0<R\le \infty$ and  $u \in \mathcal{D}^{m,p}_{R}(\alpha)$ and $\epsilon>0$. For $\epsilon>0$, let us define  the dilations  
$$u_{\epsilon}(r) = \epsilon^{-\beta} u(\epsilon^{-1}r),\;\;\mbox{for all}\;\; 0<r\le \epsilon R.$$
If $\beta=(\theta+1)/p$, then  $u_{\epsilon}\in \mathcal{D}^{m,p}_{\epsilon R}(\alpha)$ and we have
\begin{enumerate}
    \item [$(a)$]
    $
   \displaystyle \int_{0}^{\epsilon R}|u_{\epsilon}(r)|^{p^*}r^{\theta}dr=\int_{0}^{R}|u(r)|^{p^*}r^{\theta}dr
    $
    \item [$(b)$]
    $
    \displaystyle\int_{0}^{\epsilon R}|\nabla^{m}_{\alpha} u_{\epsilon}|^{p}r^{\alpha}dr=\int_{0}^{R}|\nabla^{m}_{\alpha} u|^{p}r^{\alpha}dr.
    $
\end{enumerate}
\end{lemma}
\begin{proof}
$(a)$. The change of variable $s=\epsilon^{-1}r$ yields
    \begin{equation*}
       \int_{0}^{\epsilon R}|u_{\epsilon}(r)|^{p^*}r^{\theta}dr=\int_{0}^{R}|u(s)|^{p^*}s^{\theta}dr.
    \end{equation*}
In order to get $(b)$, we firstly observe that for any $k\in\mathbb N$ and $\beta\in\mathbb{R}$
\begin{equation}\label{c0-e61}
    \Delta_{\alpha}^{k} u_{\epsilon} (r) = \epsilon^{-(\beta + 2k)} \Delta^{k}_{\alpha} u (\epsilon^{-1}r)
\end{equation}
and 
\begin{equation}\label{c0-e62}
    \big(\Delta_{\alpha}^{k} u_{\epsilon}\big)^{\prime}(r)  = \epsilon^{-(\beta + 2k+1)} \big( \Delta^{k}_{\alpha} u )^{\prime} (\epsilon^{-1}r).
\end{equation}
If $m=2k$,  from \eqref{c0-e61} and  setting $s=\epsilon^{-1} r$, we have
\begin{equation*}
    \begin{aligned}
        \int_{0}^{\epsilon R}|\nabla^{m}_{\alpha} u_{\epsilon}|^{p}r^{\alpha}dr &=\epsilon^{-p(\frac{\theta+1}{p^*}+m)}\int_{0}^{\epsilon R}|\Delta^{k}_{\alpha} u (\epsilon^{-1}r)|^{p}r^{\alpha}dr\\
        &=\epsilon^{-(\alpha+1)}\int_{0}^{\epsilon R}|\Delta^{k}_{\alpha} u (\epsilon^{-1}r)|^{p}r^{\alpha}dr\\
         &=\int_{0}^{R}|\Delta^{k}_{\alpha} u (s)|^{p}s^{\alpha}ds.
    \end{aligned}
\end{equation*}
Similarly, if  $m=2k+1$, from \eqref{c0-e62} we obtain
\begin{equation*}
    \begin{aligned}
        \int_{0}^{\epsilon R}|\nabla^{m}_{\alpha} u_{\epsilon}|^{p}r^{\alpha}dr &=\epsilon^{-p(\frac{\theta+1}{p^*}+m)}\int_{0}^{\epsilon R}|(\Delta^{k}_{\alpha} u )^{\prime}(\epsilon^{-1}r)|^{p}r^{\alpha}dr\\
         &=\int_{0}^{R}|(\Delta^{k}_{\alpha} u)^{\prime} (s)|^{p}s^{\alpha}ds.
    \end{aligned}
\end{equation*}
\end{proof}
\begin{lemma}\label{a0-prop1}  $\mathcal{S}(m,p,\alpha, \theta, R)$ is independent of $R>0$.
\end{lemma}
\begin{proof}
Fix $R>0$. For any $u \in \mathcal{D}^{m,p}_{R}(\alpha)$ and $\epsilon>0$, from Lemma~\ref{remark dilation invariance}  we can take $u_{\epsilon}\in \mathcal{D}^{m,p}_{\epsilon R}(\alpha)$ given by
 \begin{equation}\nonumber
     u_{\epsilon}(r)=\epsilon^{-\frac{\theta+1}{p^*}} u(\epsilon^{-1}r).
 \end{equation}
 Also, if $\|u\|_{L^{p^*}_{\theta}}=1$ then $\|u_{\epsilon}\|_{L^{p^*}_{\theta}}=1$. Thus, from Lemma~\ref{remark dilation invariance} 
 \begin{equation}\nonumber
 \begin{aligned}
  \mathcal{S}(m,p,\alpha, \theta, \epsilon R) & \le  \int_{0}^{\epsilon R}|\nabla_{\alpha}^{m} u_{\epsilon}|^{p}r^{\alpha}dr=\int_{0}^{ R}|\nabla_{\alpha}^{m} u|^{p}r^{\alpha}dr.
  \end{aligned} 
\end{equation}
Since $u \in \mathcal{D}^{m,p}_{R}(\alpha)$ with  $\|u\|_{L^{p^*}_{\theta}}=1$ has been taken arbitrary,  the above inequality implies  
$$\mathcal{S}(m,p,\alpha, \theta, \epsilon R) \leq \mathcal{S}(m,p,\alpha, \theta,R)$$ 
for any $R>0$ and $\epsilon>0$. On the other hand, for any $v \in \mathcal{D}^{m,p}_{\epsilon R}(\alpha),\; \epsilon>0$ we can choose  $v_{\epsilon} \in \mathcal{D}^{m,p}_{R}(\alpha)$ given by
\begin{equation}\nonumber
     v_{\epsilon}(r)=\epsilon^{\frac{\theta+1}{p^*}} v(\epsilon r).
 \end{equation}
 Also,  $\|v\|_{L^{p^*}_{\theta}}=1$ forces $\|v_{\epsilon}\|_{L^{p^*}_{\theta}}=1$. Hence, 
\begin{equation}\nonumber
 \begin{aligned}
  \mathcal{S}(m,p,\alpha, \theta, R) & \le  \int_{0}^{R}|\nabla_{\alpha}^{m} v_{\epsilon}|^{p}r^{\alpha}dr=\int_{0}^{ \epsilon R}|\nabla_{\alpha}^{m} v|^{p}r^{\alpha}dr.
  \end{aligned} 
\end{equation}
It follows that 
$$\mathcal{S}(m,p,\alpha, \theta, R) \leq \mathcal{S}(m,p,\alpha, \theta,\epsilon R).$$ 
Consequently
$$\mathcal{S}(m,p,\alpha, \theta, R) = \mathcal{S}(m,p,\alpha, \theta,\epsilon R),\;\;\mbox{for any}\;\; \epsilon>0$$
which proves the result.
\end{proof}
\begin{lemma} Let $h \in C_{c}^{\infty}(0, \infty)$ and $\varphi \in \mathcal{D}^{m,2}_{R}(\alpha)$, with $0<R\leq \infty$, then
\begin{equation}\label{claim1}
    \nabla^{m}_{\alpha}(h \varphi) = h \nabla^{m}_{\alpha} \varphi + F_{\varphi},
\end{equation}
where $F_{\varphi}=F_{\varphi}(r)$ is a linear combination of derivatives of $\varphi$ with order strictly less than $m$ involving the derivatives of $h$ with order less than or equal to $m$. In particular, if $u_k \rightharpoonup 0$ in $\mathcal{D}^{m,2}_{\infty}(\alpha)$, then
\begin{equation}\label{claim2}
    \lim_{k \to \infty} \int_{\mathrm{supp} (h)}|\nabla^{m}_{\alpha} (hu_k)|^{2} r^{\alpha} \, dr
= \lim_{k \to \infty} \int_{\mathrm{supp} (h)} h^{2}|\nabla^{m}_{\alpha} u_k|^{2} r^{\alpha} \, dr.
\end{equation}
\end{lemma}
\begin{proof}
We proceed by induction argument. First we analyze the case  $m=2k$. If $k=1$, it is easy to verify that
\begin{align*}
    \nabla^{2}_{\alpha}(h\varphi)&= h\Delta_{\alpha}\varphi+ \varphi\Delta_{\alpha}h+2\varphi^{\prime}h^{\prime}=h\nabla^{2}_{\alpha}\varphi + \varphi\Delta_{\alpha}h+2\varphi^{\prime}h^{\prime}
\end{align*}
and, thus we can take $F_{\varphi} = 2h'\varphi' + \varphi \Delta_{\alpha}h$.  Assume that \eqref{claim1} holds  on $\mathcal{D}^{2j, 2}_{R}(\alpha)$ for any integer number $j$, $1\le j<k$. That is,  for $h \in C_{c}^{\infty}(0, \infty)$ and $\varphi \in \mathcal{D}^{2j,2}_{R}(\alpha)$
we have
\begin{equation}\label{h-inductionG}
    \nabla^{2j}_{\alpha}(h \varphi) = h \nabla^{2j}_{\alpha} \varphi + G_{\varphi},
\end{equation}
where $G_{\varphi}$ is a linear combination of derivatives of $\varphi$ with order strictly less than $2j$, for  $1\le j<k$. Let $u\in \mathcal{D}^{2k, 2}_{R}(\alpha)$ and $h \in C_{c}^{\infty}(0, \infty)$. Then, from \eqref{h-inductionG}
\begin{align*}
    \Delta^{k}_{\alpha}(h \varphi) &=\Delta_{\alpha}[\Delta^{k-1}_{\alpha}(h \varphi)] 
=\Delta_{\alpha}[h \Delta^{k-1}_{\alpha} \varphi + G_{\varphi}] \\
&= \Delta_{\alpha}(h \Delta^{k-1}_{\alpha} \varphi) + \Delta_{\alpha}(G_{\varphi}) 
\end{align*}
where $G_{\varphi}$ is a linear combination of derivatives of $\varphi$ with order strictly less than $2k-2$. It follows that
\begin{eqnarray*}
\Delta^{k}_{\alpha}(h \varphi) &=& h \Delta^{k}_{\alpha} \varphi + 2 h' (\Delta^{k-1}_{\alpha} \varphi)'+ (\Delta_{\alpha}h)\Delta^{k-1}_{\alpha} \varphi+ \Delta_{\alpha}(G_{\varphi}). 
\end{eqnarray*}
Taking into account the expansions in Lemma~\ref{c0-lema3}, from the above identity we can see that \eqref{claim1} holds for $j=k$ with $$F_{\varphi} =2 h' (\Delta^{k-1}_{\alpha} \varphi)'+ (\Delta_{\alpha}h)\Delta^{k-1}_{\alpha} \varphi+ \Delta_{\alpha}(G_{\varphi}).$$
Similarly, we can see that \eqref{claim1} holds for the case $m=2k+1$.
 
Next, we will prove \eqref{claim2}. From \eqref{claim1}, we obtain
$$\|\nabla^{m}_{\alpha}(h u_k) - h \nabla^{m}_{\alpha} u_{k}\|^{2}_{L^{2}_{\alpha}(\mathrm{supp}(h))} = \int_{\mathrm{supp}(h)} |F_{u_{k}}|^{2} r^{\alpha} \, dr,$$
for any $h \in C_{c}^{\infty}(0, \infty)$.
It is enough to show that $F_{u_k} \to 0$ em $L^{2}_{\alpha}(\mathrm{supp} (h))$.
Indeed, for any $l=0, 1, \cdots, m-1$, we define $A_{l} : \mathcal{D}^{m,2}_{\infty}(\alpha) \to \mathcal{D}^{m-l,2}_{\infty}(\alpha)$ by $A_{l}(u) = u^{(l)}$. Of course, $A_{l}$ is a linear operator satisfying 
$$\|A_{l}(u)\|_{\mathcal{D}^{m-l,2}_{\infty}(\alpha)} = \|[A_{l}(u)]^{(m-l)}\|_{L^{2}_{\alpha}(0, \infty)} = \|u^{(m)}\|_{L^{2}_{\alpha}(0, \infty)} = \|u\|_{\mathcal{D}^{m,2}_{\infty}(\alpha)}.$$
By the continuity of  the linear operator $A_{l}$ and  $u_{k} \rightharpoonup 0$ in $\mathcal{D}^{m,2}_{\infty}(\alpha)$, we get
$u_{k}^{(l)} \rightharpoonup 0$ in $\mathcal{D}^{m-l,2}_{\infty}(\alpha)$. From \eqref{DX-Identity},  $ \mathcal{D}^{m-l,2}_{\infty}(\alpha) = X_{\infty}^{m-l,2}(\beta_{0}, \cdots, \beta_{m-l})$, with $\beta_{j} = \alpha - (m-l-j)p$. Then, $u^{(l)}_{k} \in X_{\infty}^{m-l,2}(\beta_{0}, \cdots, \beta_{m-l})$ and $ u^{(l)}_{k} \rightharpoonup 0$ in $X_{\infty}^{m-l,2}(\beta_{0}, \cdots, \beta_{m-l})$. 
Note that the restriction operator $u \to u_{|_{(0,R)}}$ is continuous from $X_{\infty}^{m-l,2}(\beta_{0}, \cdots, \beta_{m-l})$ into $W_{R}^{m-l,2}(\beta_{0}, \cdots, \beta_{m-l})$, where $R>0$ such that $\mathrm{supp}\, (h) \subseteq (0,R)$. Thus,   $ u^{(l)}_{k} \rightharpoonup 0$ in  $W_{R}^{m-l,2}(\beta_{0}, \cdots, \beta_{m-l})$. Finally, up to a subsequence, the compact embedding $W^{m-l,2}_{R}(\beta_{0}, \cdots, \beta_{m-l}) \hookrightarrow L^{2}_{\alpha}(0,R)$ (see, \cite[Theorem~1.1]{DoLuHa}) implies that
$$u_{k}^{(l)} \to 0 \quad\text{in}\quad L^{2}_{\alpha}(0,R), \quad \forall\,\, l = 0,1, \cdots , m-1.$$
Then, 
\begin{eqnarray*}
\int_{\mathrm{supp}\, (h)} |F_{u_{k}}|^{2} r^{\alpha} \, dr &\leq& \int_{0}^{R} |F_{u_{k}}|^{2} r^{\alpha} \, dr\\
&\leq& C \sum^{m-1}_{l=0} \int_{0}^{R} |u_{k}^{(l)}|^{2} r^{\alpha} \, dr \to 0
\end{eqnarray*}
which completes the proof.
\end{proof}
\noindent Let $\mathcal{M}(0,\infty)$ the space of Radon measures on $(0,\infty)$ which can be identiﬁed with the dual of space $C_0(0,\infty)$, the completion of $C_c(0,\infty)$ under the norm $\|\varphi\|_{\infty}=\max_{r\in (0,\infty)}|\varphi(r)|$, where $C_c(0,\infty)$ is the space of continuous functions with compact support in $(0,\infty)$. 
\begin{lemma}\label{a0-prop2}
Let $(u_{k})$ be a sequence in $\mathcal{D}^{m,2}_{\infty}(\alpha)$ satisfying the conditions:
\begin{itemize}
    \item [$(i)$] $u_k \rightharpoonup u$ in $\displaystyle\mathcal{D}^{m,2}_{\infty}(\alpha)$
    
    \item [$(ii)$] $\displaystyle|\nabla^{m}_{\alpha} (u_{k} - u)|^{2} r^{\alpha} dr 
    \rightharpoonup \mu$ in $\displaystyle\mathcal{M}(0,\infty)$
    
    \item [$(iii)$] $\displaystyle|u_{k} - u|^{2^{*}} r^{\theta} dr \rightharpoonup \zeta$ in $\mathcal{M}(0,\infty)$
    
    \item [$(iv)$] $u_{k}(r) \to u(r)$ a.e. on $(0,\infty)$
\end{itemize}
and define
$$\mu_{\infty} = \lim_{L \to \infty} \left(\limsup_{k \to \infty} \int_{L}^{\infty} |\nabla_{\alpha}^{m} u_k|^{2} r^{\alpha} \,dr\right)$$
and 
$$\zeta_{\infty} = \lim_{L \to \infty} \left(\limsup_{k \to \infty} \int_{L}^{\infty} |u_k|^{2^{*}} r^{\theta} \,dr\right).$$
Then
\begin{equation}\label{a0-e1}
    \|\zeta\|^{2/2^{*}} \leq \mathcal{S}^{-1}\|\mu\|,
\end{equation}
\begin{equation}\label{a0-e2}
    (\zeta_{\infty})^{2/2^{*}} \leq \mathcal{S}^{-1} \mu_{\infty},
\end{equation}
\begin{equation}\label{a0-e3}
    \limsup_{k \to \infty} \|\nabla^{m}_{\alpha} u_k\|^{2}_{L^{2}_{\alpha}} = \|\nabla^{m}_{\alpha} u\|^{2}_{L^{2}_{\alpha}} + \|\mu\| + \mu_{\infty},
\end{equation}
\begin{equation}\label{a0-e4}
    \limsup_{k \to \infty} \|u_k\|^{2^{*}}_{L^{2^{*}}_{\theta}} = \| u\|^{2^{*}}_{L^{2^{*}}_{\theta}} + \|\zeta\| + \zeta_{\infty}.
\end{equation}
Furthermore, if $u=0$ and $ \|\zeta\|^{2/2^{*}} = \mathcal{S}^{-1}\|\mu\|$, then $\mu$ and $\zeta$ are concentrated at a single point.
\end{lemma}
\begin{proof} Firstly, we suppose
$u=0$. Thus, for any  $h \in C_{c}^{\infty}(0, \infty)$,  Theorem~\ref{c0-t2} and \eqref{claim2} yield
\begin{eqnarray*}
\left(\int_{0}^{\infty} |h|^{2^{*}} \,d\zeta \right)^{2/2^{*}} &=& \lim_{k \to \infty} \left(\int_{\mathrm{supp}\,(h)} |h u_{k}|^{2^{*}} r^{\theta} \,dr \right)^{2/2^{*}}  \\
&\leq& \mathcal{S}^{-1} \lim_{k \to \infty} \int_{\mathrm{supp}\, (h)} |\nabla^{m}_{\alpha} (hu_k)|^{2} r^{\alpha} \, dr\\
&=& \mathcal{S}^{-1}\lim_{k \to \infty} \int_{\mathrm{supp}\, (h)} h^{2}|\nabla^{m}_{\alpha} u_k|^{2} r^{\alpha} \, dr.
\end{eqnarray*}
It follows that 
\begin{equation}\label{a0-e5}
    \left(\int_{0}^{\infty} |h|^{2^{*}} \,d\zeta \right)^{2/2^{*}} \leq \mathcal{S}^{-1} \int^{\infty}_{0} h^{2} \, d\mu, \quad \forall\,\, h \in C_{c}^{\infty}(0, \infty).
\end{equation}
Now, for any $L>1$, let $\psi_{L} \in C^{\infty}_{c}(0,\infty)$ with $0\leq \psi_{L} \leq 1$ such that $\psi_{L}\equiv 0$ on $(0, L)$ and $\psi_{L}\equiv 1$ on $(L+1,\infty)$. By Theorem \ref{c0-t2}, we have
\begin{align*}
&\left(\int_{0}^{\infty} |\psi_{L} u_{k}|^{2^{*}} r^{\theta} \,dr \right)^{2/2^{*}} \le \mathcal{S}^{-1} \int_{0}^{\infty} |\nabla^{m}_{\alpha}(\psi_{L} u_{k})|^{2} r^{\alpha}\, dr\\
&= \mathcal{S}^{-1} \left( \int_{L}^{L+1} |\nabla^{m}_{\alpha}(\psi_{L} u_{k})|^{2} r^{\alpha}\, dr + \int_{L+1}^{\infty} |\nabla^{m}_{\alpha}u_{k}|^{2} r^{\alpha}\, dr\right).
\end{align*}
Arguing as in \eqref{claim2}, we can show that
$$\lim_{k \to \infty} \int_{L}^{L+1} |\nabla^{m}_{\alpha}(\psi_{L} u_{k})|^{2} r^{\alpha}\, dr = \lim_{k \to \infty} \int_{L}^{L+1} \psi_{L}^{2}|\nabla^{m}_{\alpha} u_{k}|^{2} r^{\alpha}\, dr.$$
Hence 
\begin{equation}\label{a0-e6}
    \limsup_{k \to \infty} \left(\int_{0}^{\infty} |\psi_{L} u_{k}|^{2^{*}} r^{\theta} \,dr \right)^{2/2^{*}} \leq \mathcal{S}^{-1} \limsup_{k \to \infty} \int_{L}^{\infty} \psi_{L}^{2}|\nabla^{m}_{\alpha} u_{k}|^{2} r^{\alpha}\, dr.
\end{equation} 
 We claim that 
\begin{equation}\label{claim3}
\begin{aligned}
    &\zeta_{\infty} = \lim_{L \to \infty}\limsup_{k \to \infty} \int_{0}^{\infty} |\psi_{L} u_{k}|^{2^{*}} r^{\theta} \,dr\\
&\mu_{\infty} = \lim_{L \to \infty} \limsup_{k \to \infty} \int_{L}^{\infty} \psi_{L}^{2}|\nabla^{m}_{\alpha} u_{k}|^{2} r^{\alpha}\, dr.
\end{aligned}
\end{equation}
Indeed, note that
$$\int_{L+1}^{\infty} |u_{k}|^{2^{*}} r^{\theta} \,dr \leq \int_{0}^{\infty} \psi_{L}^{2^{*}} |u_{k}|^{2^{*}} r^{\theta} \,dr \leq \int_{L}^{\infty} | u_{k}|^{2^{*}} r^{\theta} \,dr$$
and
$$\int_{L+1}^{\infty} |\nabla^{m}_{\alpha} u_{k}|^{2} r^{\alpha}\, dr \leq \int_{0}^{\infty} \psi_{L}^{2}|\nabla^{m}_{\alpha} u_{k}|^{2} r^{\alpha}\, dr \leq \int_{L}^{\infty} |\nabla^{m}_{\alpha} u_{k}|^{2} r^{\alpha}\, dr.$$
Letting $k\to\infty$ and after $L\to\infty$, the two inequalities above yield \eqref{claim3}. Combining \eqref{a0-e5},  \eqref{a0-e6} and \eqref{claim3}, we obtain \eqref{a0-e1} and \eqref{a0-e2} for $u=0$.

Now, assume that the weak limit $u\in \mathcal{D}^{m,2}_{\infty}(\alpha)$ is arbitrary and set $v_k=u_k-u$. Of course we have  $v_{k} \rightharpoonup 0$ in $\mathcal{D}^{m,2}_{\infty}(\alpha)$. Moreover,   since $\mathcal{D}^{m,2}_{\infty}(\alpha)$ is a Hilbert space equipped with the inner product
$$\left< u, v \right>_{\nabla^{m}_{\alpha}} =\int_{0}^{\infty} (\nabla^{m}_{\alpha} u )(\nabla^{m}_{\alpha} v)r^{\alpha} \, dr$$
we have   $\left<u_k, u\right>_{\nabla^{m}_{\alpha}} \to \|u\|^{2}_{\nabla^{m}_{\alpha}}$ in $\mathbb{R}$. So, one has 
\begin{equation}\label{a0-e9}
    \lim_{k \to \infty} \|v_k\|^2_{\nabla^{m}_{\alpha}} = \lim_{k \to \infty}\|u_k\|^{2}_{\nabla^{m}_{\alpha}} - \|u\|^{2}_{\nabla^{m}_{\alpha}}.
\end{equation}
Thus, for any $h \in C^{\infty}_{c}(0, \infty)$, we have
\begin{align*}
&\left|\int_{0}^{\infty} h (|\nabla^{m}_{\alpha} u_k|^{2} - |\nabla^{m}_{\alpha} (u_k-u)|^{2} - |\nabla^{m}_{\alpha} u|^{2})r^{\alpha} \, dr\right|\\ 
& \le \|h\|_{\infty}\left| \|u_k\|^{2}_{\nabla^{m}_{\alpha}} - \|v_k \|^{2}_{\nabla^{m}_{\alpha}} - \|u\|^{2}_{\nabla^{m}_{\alpha}}\right|.
\end{align*}
Letting $k\to \infty$, by using \eqref{a0-e9} and $(ii)$, it follows that
\begin{equation}\label{a0-e10}
    |\nabla^{m}_{\alpha} u_k|^{2}r^{\alpha} \, dr \rightharpoonup \mu + |\nabla^{m}_{\alpha}u|^{2}r^{\alpha} \, dr \quad \text{in} \quad \mathcal{M}(0,\infty).
\end{equation}
On the other hand, the convergence $u_k \rightharpoonup u$ in $\mathcal{D}^{m,2}_{\infty}(\alpha)$ and  the embedding $\mathcal{D}^{m,2}_{\infty}(\alpha) \hookrightarrow L^{2^{*}}_{\theta}(0, \infty)$ ensure that $(u_k)$ is bounded in $L^{2^{*}}_{\theta}(0, \infty)$. Recalling $(iv)$ we are in a position to use the Brezis-Lieb Lemma to get
$$
\int_{0}^{\infty}h|u|^{2^*}r^{\theta}dr=\lim_{k\to \infty}\left(\int_{0}^{\infty}h|u_k|^{2^*}r^{\theta}dr-\int_{0}^{\infty}h|v_k|^{2^*}r^{\theta}dr\right)
$$
and consequently $(iii)$ yields
\begin{equation}\label{a0-e11}
    | u_k|^{2^{*}}r^{\theta} \, dr \rightharpoonup \zeta + |u|^{2^{*}}r^{\theta} \, dr \quad \text{in} \quad \mathcal{M}(0,\infty).
\end{equation}
By using $v_k$ instead of $u_k$ we can apply the same argument of the case $u=0$ to obtain the estimate \eqref{a0-e5} from which we obtain \eqref{a0-e1}. In order to prove \eqref{a0-e2} we  first note that 
$$\int_{L}^{\infty} |\nabla^{m}_{\alpha} v_k|^{2} r^{\alpha} \, dr = \int_{L}^{\infty} (|\nabla^{m}_{\alpha} u_k|^{2} - 2 (\nabla^{m}_{\alpha} u_k) (\nabla^{m}_{\alpha} u) + |\nabla^{m}_{\alpha} u|^{2}) r^{\alpha} \, dr$$
implies 
$$\limsup_{k \to \infty} \int_{L}^{\infty} |\nabla^{m}_{\alpha} v_k|^{2} r^{\alpha} \, dr = \limsup_{k \to \infty} \int_{L}^{\infty} |\nabla^{m}_{\alpha} u_k|^{2} r^{\alpha} \, dr - \int_{L}^{\infty} |\nabla^{m}_{\alpha} u|^{2}r^{\alpha} \,dr.$$
Thus, $\displaystyle\lim_{L \to \infty} \int_{L}^{\infty} |\nabla^{m}_{\alpha} u|^{2}  r^{\alpha}\, dr = 0$ implies
\begin{equation}\label{a0-e12}
    \lim_{L \to \infty}\limsup_{k \to \infty} \int_{L}^{\infty} |\nabla^{m}_{\alpha} v_k|^{2} r^{\alpha} \, dr = \lim_{L \to \infty}\limsup_{k \to \infty} \int_{L}^{\infty} |\nabla^{m}_{\alpha} u_k|^{2} r^{\alpha} \, dr= \mu_{\infty}.
\end{equation}
On the other hand, Brezis-Lieb Lemma yields
$$\lim_{k \to \infty} \left( \int_{L}^{\infty} |u_k|^{2^{*}} r^{\theta}\, dr- \int_{L}^{\infty}|v_k|^{2^{*}} r^{\theta}\, dr\right) = \int_{L}^{\infty} |u|^{2^{*}} r^{\theta}\, dr.$$
Since $\displaystyle\lim_{L \to \infty} \int_{L}^{\infty} |u|^{2^{*}} r^{\theta}\,dr = 0$, we also have
\begin{equation}\label{a0-e13}
    \lim_{L \to \infty} \limsup_{k \to \infty} \int_{L}^{\infty}|v_k|^{2^{*}} r^{\theta}\, dr = \lim_{L \to \infty} \limsup_{k \to \infty} \int_{L}^{\infty} |u_k|^{2^{*}} r^{\theta}\, dr =  \zeta_{\infty}.
\end{equation}
By using $v_k$ instead of $u_k$, from \eqref{a0-e12} and \eqref{a0-e13}, we can apply the same argument of the case $u=0$ to obtain \eqref{a0-e6} and \eqref{claim3} for any weak limit $u$. Thus, \eqref{a0-e2} holds.

Next, we will prove that \eqref{a0-e3} and  \eqref{a0-e4}. For any $L>1$, using
$|\nabla^{m}_{\alpha}u_k|^{2} = \psi_{L}^{2}|\nabla^{m}_{\alpha}u_k|^{2} + (1- \psi_{L}^{2})|\nabla^{m}_{\alpha}u_k|^{2}$
and \eqref{a0-e10}, we have
\begin{equation}\label{a0-e14}
\begin{aligned}
\limsup_{k \to \infty} \int_{0}^{\infty} |\nabla^{m}_{\alpha}u_k|^{2} r^{\alpha} \,\ dr &= \limsup_{k \to \infty} \int_{0}^{\infty} \psi_{L}^{2}|\nabla^{m}_{\alpha}u_k|^{2} r^{\alpha} \,\ dr + \int_{0}^{\infty}(1-\psi_{L}^{2})\,d\mu\\
& +  \int_{0}^{\infty}(1-\psi_{L}^{2})|\nabla^{m}_{\alpha}u|^{2}r^{\alpha}\,d\mu.
\end{aligned}
\end{equation}
Letting $L \to \infty$ in \eqref{a0-e14} and using Lebesgue's dominated convergence theorem, we have
\begin{align*}
\limsup_{k \to \infty} \int_{0}^{\infty} |\nabla^{m}_{\alpha}u_k|^{2} r^{\alpha} \,\ dr &= \lim_{L \to \infty}\limsup_{k \to \infty} \int_{0}^{\infty} \psi_{L}^{2}|\nabla^{m}_{\alpha}u_k|^{2} r^{\alpha} \,\ dr + \int_{0}^{\infty}\,d\mu\\
& + \int_{0}^{\infty}|\nabla^{m}_{\alpha}u|^{2}r^{\alpha}\,d\mu.
\end{align*}
By using \eqref{claim3}, we obtain \eqref{a0-e3}. To prove \eqref{a0-e4}, we can proceed analogously  with the help of \eqref{a0-e11}.

Finally, assume $u=0$ and $\|\zeta\|^{2/2^{*}} = S^{-1} \|\mu\|$. We will show that the measures $\mu$ and $\zeta$ are concentrated in a singular point. Indeed, for any $h \in C_{c}^{\infty}(0, \infty)$ the H\"{o}lder inequality implies that
\begin{eqnarray*}
\int_{0}^{\infty} h^{2}\, d\mu &\leq& \left(\int_{\mathrm{supp}(h)} d\mu \right)^{\frac{2m+\theta-\alpha}{\theta+1}}\left(\int_{0}^{\infty} |h|^{2^{*}} \, d\mu \right)^{\frac{2}{2^{*}}}.
\end{eqnarray*}
Thus, 
\begin{equation}\label{a0-e7}
    \left(\int_{0}^{\infty} h^{2}\, d\mu \right)^{\frac{1}{2}} \leq \|\mu \|^{\frac{2m+\theta-\alpha}{2(\theta+1)}}\left( \int_{0}^{\infty} |h|^{2^{*}} \, d\mu \right)^{\frac{1}{2^{*}}}.
\end{equation}
Combining \eqref{a0-e5} and \eqref{a0-e7}, we have
$$\left(\int_{0}^{\infty} |h|^{2^{*}}  d\zeta \right)^{\frac{1}{2^{*}}} \leq \mathcal{S}^{-\frac{1}{2}}\|\mu \|^{\frac{2m+\theta-\alpha}{2(\theta+1)}}\left( \int_{0}^{\infty} |h|^{2^{*}} \, d\mu \right)^{\frac{1}{2^{*}}}.$$
Therefore, $\zeta \leq \mathcal{S}^{-\frac{2^{*}}{2}} \|\mu\|^{\frac{2m + \theta-\alpha}{\alpha-2m +1}} \mu$. Since that $\|\zeta\|^{\frac{2}{2^{*}}} = \mathcal{S}^{-1} \|\mu\|$, we have
\begin{equation}\label{a0-e8}
    \zeta = \mathcal{S}^{-\frac{2^{*}}{2}} \|\mu\|^{\frac{2m + \theta-\alpha}{\alpha-2m +1}} \mu.
\end{equation}
Using \eqref{a0-e8} in \eqref{a0-e5}, we have
$$\left(\int_{0}^{\infty} h^{2} \, d\mu \right)^{\frac{1}{2}} \geq \|\mu\|^{\frac{2m+\theta-\alpha}{2(\theta+1)}}\left( \int_{0}^{\infty} |h|^{2^{*}} \, d\mu \right)^{\frac{1}{2^{*}}}.$$
In particular, for all open set $A\subset(0, \infty)$ holds
$$\mu(A)^{\frac{1}{2}} \geq \|\mu\|^{\frac{2m+\theta-\alpha}{2(\theta+1)}} \mu(A)^{\frac{1}{2^{*}}}.$$
Thus,
\begin{equation}\label{mA>mnorm}
   \mu(A) \geq \|\mu\|.
\end{equation}
It follows that $\mu$ is concentrated in a single point. From \eqref{a0-e8}, the same conclusion holds for $\zeta$.
\end{proof}
\begin{remark}\label{R4} For any $u \in \mathcal{D}^{m,p}_{\infty}(\alpha)\setminus\{0\}$  its translation 
$u^{s}(r) = u(r+s)$, with $s\geq 0$ does not satisfy
$$\|  u^{s} \|_{L_{\theta}^{p^{*}}} =  \|  u \|_{L_{\theta}^{p^{*}}} \quad \text{and}\quad \| \nabla^{m}_{\alpha} u^{s} \|_{L_{\alpha}^{p}} = \| \nabla^{m}_{\alpha} u \|_{L_{\alpha}^{p}},\;\; \forall s>0.$$ 
Thus, the problem \eqref{c0-e40} is not invariant by translations. In particular, we can not use the re-scaled function family $v^{s,\lambda}(r) = \lambda^{\frac{\theta+1}{2^{*}}} u(\lambda r+s)$, $\lambda, s>0$
to generate suitable minimizing sequences for $\mathcal{S}$. Fortunately, we can do it with $s=0$, see Remark~\ref{remark dilation invariance}.
\end{remark}

\begin{proof}[Proof of Theorem \ref{a0-teo1}:] Let $(u_{k})\subset D^{m,2}_{\infty}(\alpha)$ be a minimizing sequence for $\mathcal{S}$ such that 
\begin{equation}\label{u-min}
    \|u_k\|_{L^{2^{*}}_{\theta}}=1\;\;\mbox{and}\;\; \mathcal{S}(m,2,\alpha, \theta, \infty)=\lim_{k\to\infty}\|\nabla^{m}_{\alpha} u_{k}\|^{2}_{L^{2}_{\alpha}}.
\end{equation}
For $k\in \mathbb{N}$, let $Q_{k}: (0,\infty) \to \mathbb{R}$ be given by
$$Q_{k}(t) = \int^{t}_{0} |u_{k}(r)|^{2^{*}} r^{\theta} \, dr.$$
For each $k$,  $Q_k$ is a continuous function and satisfies
$$\lim_{t \to 0^{+}} Q_{k}(t) = 0, \quad \lim_{t \to \infty} Q_{k}(t) = 1 \quad \text{and} \quad Q_{k}(t_{1})\leq Q_{k}(t_{2}), \quad\text{if} \quad t_{1} \leq t_{2}.$$
Therefore,  there is a real sequence $(t_k)\subset(0,\infty)$ such that
\begin{equation}\label{a0-e15}
    Q_{k}(t_{k}) = \frac{1}{2}.
\end{equation}
Now, let us define $v_{k}(r) = t_{k}^{\frac{\theta+1}{2^{*}}} u_{k}(rt_{k})$. From Remark~\ref{remark dilation invariance}, we have
\begin{equation}\label{v-min}
    \|v_{k}\|_{L^{2^{*}}_{\theta}} = \|u_{k}\|_{L^{2^{*}}_{\theta}}=1 \quad\text{and}\quad \|\nabla^{m}_{\alpha} v_{k}\|_{L^{2}_{\alpha}} = \|\nabla^{m}_{\alpha} u_{k}\|_{L^{2}_{\alpha}}.
\end{equation}
In addition, by using the change of variable $z= rt_{k}$ and \eqref{a0-e15}, we have 
\begin{equation}\label{a0-e17}
    \int_{0}^{1} |v_{k}(r)|^{2^{*}}r^{\theta}\, dr = \int_{0}^{t_{k}} |u_{k}(r)|^{2^{*}}r^{\theta}\, dr=\frac{1}{2}.
\end{equation}
Since that $(v_{k})$ is a bounded sequence in  the Hilbert space $\mathcal{D}^{m,2}_{\infty}(\alpha)$, up to a subsequence, we may assume 
\begin{equation}\label{list-limit}
   \begin{aligned}
&v_{k} \rightharpoonup v & &\mbox{in}\; \mathcal{D}^{m,2}_{\infty}
 (\alpha)\\   
&|\nabla^{m}_{\alpha}(v_k-v)|^{2}r^{\alpha} \rightharpoonup \mu &  & \mbox{in}\;  \mathcal{M}(0,\infty)\\
&|v_k-v|^{2^{*}}r^{\theta} \rightharpoonup \zeta &  &\mbox{in}\;  \mathcal{M}(0,\infty)\\
& v_k(r)\rightarrow v(r) & &\mbox{a.e on}\; (0,\infty).
\end{aligned} 
\end{equation}
Combining Lemma~\ref{a0-prop2} with \eqref{u-min} and \eqref{v-min} we can write
\begin{equation}\label{a0-e18}
    \mathcal{S} = \|\nabla^{m}_{\alpha} v\|^{2}_{L^{2}_{\alpha}} + \|\mu\| + \mu_{\infty} 
\end{equation}
and
\begin{equation}\label{a0-e19}
    1 = \|v\|^{2^{*}}_{L^{2^{*}}_{\theta}} +\|\zeta\|+ \zeta_{\infty}.
\end{equation}
Using \eqref{c0-t2-eq}, \eqref{a0-e1} and  \eqref{a0-e2}, we have
\begin{equation}\label{a0-e20+}
\mathcal{S}\ge\mathcal{S}\left(\left(\|v\|^{2^{*}}_{L^{2^{*}}_{\theta}}\right)^{\frac{2}{2^{*}}} + \|\zeta\|^{\frac{2}{2^{*}}} + (\zeta_{\infty})^{\frac{2}{2^{*}}}\right).
\end{equation}
It follows from \eqref{a0-e19} and \eqref{a0-e20+} that $\|v\|^{2^{*}}_{L^{2^{*}}_{\theta}}, \|\zeta\|$ and $\zeta_{\infty}$  are equal either to $0$ or to $1$. From \eqref{u-min}, \eqref{v-min} and \eqref{list-limit}, we have 
$$\|\nabla_{\alpha}^{m} v\|^{2}_{L_{\alpha}^{2}} \leq \lim_{k \to \infty} \|\nabla_{\alpha}^{m} v_{k}\|^{2}_{L_{\alpha}^{2}} = \mathcal{S}.$$
Then, $v$ is a minimizer provided that  $\|v\|^{2^{*}}_{L^{2^{*}}_{\theta}}=1$. From \eqref{a0-e19}, it remains to prove that $$\zeta_{\infty}=\|\zeta\|=0.$$
But, from \eqref{v-min} and \eqref{a0-e17}, for all $L>1$ we have
$$\int_{L}^{\infty} |v_{k}|^{2^{*}} r^{\theta} \, dr \leq \int_{1}^{\infty} |v_{k}|^{2^{*}} r^{\theta} \, dr = \frac{1}{2}.$$
So,  $\zeta_{\infty} \leq \frac{1}{2}$ and consequently  $\zeta_{\infty}=0$. By contradiction, suppose that $\|\zeta\| = 1$. Hence, \eqref{a0-e19} yields $\|v\|^{2^{*}}_{L^{2^{*}}_{\theta}} = \zeta_{\infty} = 0$ and consequently $v=0$. In addition, from \eqref{a0-e18} we have $$\|\mu\| \mathcal{S}^{-1} \leq 1 = \|\zeta\|^{\frac{2}{2^{*}}}.$$ Hence, by using \eqref{a0-e1} we obtain $\|\mu\| \mathcal{S}^{-1}  = \|\zeta\|^{\frac{2}{2^{*}}}$. From Lemma \ref{a0-prop2} (cf. \eqref{a0-e8} and \eqref{mA>mnorm}), for any open set $A\subset(0,\infty)$ we have 
\begin{equation}\label{zA>1}
  \begin{aligned}
    \zeta(A)&=\mathcal{S}^{-\frac{2^{*}}{2}} \|\mu\|^{\frac{2m + \theta-\alpha}{\alpha-2m +1}} \mu(A) \\
    & \ge \mathcal{S}^{-\frac{2^{*}}{2}} \|\mu\|^{\frac{2m + \theta-\alpha}{\alpha-2m +1}}\|\mu\|\\
    &=\mathcal{S}^{-\frac{2^{*}}{2}} \|\mu\|^{\frac{\theta+1}{\alpha-2m +1}}\\
    &=\mathcal{S}^{-\frac{2^{*}}{2}} \|\mu\|^{\frac{2^*}{2}}=\|\zeta\|=1.
\end{aligned}  
\end{equation}
But, letting $k\to\infty$ in \eqref{a0-e17} we obtain 
$$
\zeta(A_0)=\frac{1}{2}<\|\zeta\|=1, \;\mbox{with}\;\; A_0=(0,1)
$$
which contradicts \eqref{zA>1}. Thus, $\|\zeta\|=0$ and the proof is completed.  
\end{proof}

\begin{proof}[Proof of Corollary~\ref{d1}]
Let us consider the  functional $I: \mathcal{D}^{m,2}_{\infty}(\alpha) \to \mathbb{R}$ given by
$$I(u) = \frac{1}{2} \int_{0}^{\infty} |\nabla^{m}_{\alpha} u|^{2} r^{\alpha} \,dr - \dfrac{1}{2^{*}} \int_{0}^{\infty} |u|^{2^{*}} r^{\theta}\, dr.$$
From  Theorem \ref{c0-t2}, we can see that $I\in C^{1}(\mathcal{D}^{m,2}_{\infty}(\alpha),\mathbb{R})$ with
$$I'(u).v = \int_{0}^{\infty} \nabla^{m}_{\alpha} u \nabla^{m}_{\alpha} v r^{\alpha} \,dr -  \int_{0}^{\infty} |u|^{2^{*}-2} u v r^{\theta}\, dr, \quad \forall \,u,v \in \mathcal{D}^{m,2}_{\infty}(\alpha).$$
Hence, the critical points of $I$ are the weak solutions to the equation \eqref{c0-e8}. Now, let us consider the functionals $F, G\in C^{1}(\mathcal{D}^{m,2}_{\infty}(\alpha),\mathbb{R})$ given by 
$$G(u) = \int_{0}^{\infty} |\nabla^{m}_{\alpha} u|^{2} r^{\alpha}\, dr \quad \text{and} \quad F(u) = \int_{0}^{\infty}|u|^{2^{*}} r^{\theta} \,dr -1. $$
By Theorem \ref{a0-teo1}, the extremal function $z\in \mathcal{D}^{m,2}_{\infty}(\alpha)$ minimizers the functional $G$ under the constraint $F(u)=0.$ Thus, the Lagrange Multiplier Theorem yields 
\begin{equation}\label{a0-e20}
    2\int_{0}^{\infty} \nabla^{m}_{\alpha} z \nabla^{m}_{\alpha} v r^{\alpha}\, dr = \lambda 2^{*}\int_{0}^{\infty} |z|^{2^{*}-2} z v r^{\theta}\, dr, \quad \mbox{for all}\;\; v \in \mathcal{D}^{m,2}_{\infty}(\alpha),
\end{equation}
for some $\lambda\in\mathbb{R}$. By choosing $v=z$ in \eqref{a0-e20}, we have $2^{*}\lambda = 2 \mathcal{S}$. Finally,  using \eqref{a0-e20} we can see that $z_{\mathcal{S}} = \mathcal{S}^{\frac{1}{2^{*} - 2}} z$ satisfies $I'(z_{\mathcal{S}}).v =0$ for all $v\in \mathcal{D}^{m,2}_{\infty}(\alpha)$.    
\end{proof}

\end{document}